\documentclass{amsart}
\usepackage[utf8]{inputenc}
\pdfoutput=1

\usepackage{graphicx, booktabs,enumerate,microtype}
\usepackage[dvipsnames]{xcolor}
\usepackage{amssymb,amsmath,amsthm, stmaryrd, mathtools} 
\usepackage{tikz, tikz-cd}
\usetikzlibrary {graphs}
\usepackage{hyperref}
\usepackage{mathdots}
\usepackage[foot]{amsaddr}

\usepackage{mathrsfs}

\hypersetup{colorlinks=true, linkcolor=RoyalBlue, citecolor=JungleGreen}

\let\oldtocsection=\tocsection

\let\oldtocsubsection=\tocsubsection

\let\oldtocsubsubsection=\tocsubsubsection

\renewcommand{\tocsection}[2]{\hspace{0em}\oldtocsection{#1}{#2}}
\renewcommand{\tocsubsection}[2]{\hspace{1em}\oldtocsubsection{#1}{#2}}
\renewcommand{\tocsubsubsection}[2]{\hspace{2em}\oldtocsubsubsection{#1}{#2}}

\newcounter{prcounter}

\newcommand{\lf}{\left}
\newcommand{\ri}{\right}

\newcommand{\f}{\frac} 
 
\newcommand{\into}{\hookrightarrow}
\newcommand{\onto}{\twoheadrightarrow}

\newcommand{\wh}{\widehat}

\DeclareMathOperator{\Supp}{Supp}

\DeclareMathOperator{\tr}{tr}

\newcommand{\m}[1]{\mathbf{#1}}
\newcommand{\mf}[1]{\mathfrak{#1}}
\newcommand{\mc}[1]{\mathcal{#1}}

\newcommand{\C}{\mathbb C}
\newcommand{\R}{\mathbb R}

\newcommand{\Z}{\mathbb Z}

\newcommand{\SL}{\mathrm{SL}}
\newcommand{\Sp}{\mathrm{Sp}}

\newcommand{\GL}{\mathrm{GL}}

\newcommand{\St}{\mathrm{St}}
\newcommand{\eps}{\epsilon}
\newcommand{\om}{\omega}

\newcommand{\bs}{\backslash}
\newcommand{\1}{\m 1}

\DeclareMathOperator{\Speh}{\mathrm{Sp}}

\newcommand{\dom}{\backslash}
\newcommand{\BC}{\mathbb C}

\newcommand{\fg}{\mathfrak{g}}

\renewcommand{\th}{^{\text{th}}}

\usepackage{stackengine,scalerel}

\newtheorem{thm}{Theorem}[subsection]
\newtheorem{prop}[thm]{Proposition}

\newtheorem{cor}[thm]{Corollary}
\newtheorem{lem}[thm]{Lemma}
\newtheorem{conj}[thm]{Conjecture}

\newtheorem{ex}[thm]{Example}

\theoremstyle{remark}
\newtheorem*{note}{Note}
\newtheorem{rmk}[thm]{Remark}

\theoremstyle{definition}
\newtheorem{dfn}[thm]{Definition}
\newtheorem{conv}[thm]{Convention}

\numberwithin{equation}{subsection}

\allowdisplaybreaks

\title[Uniform Bounds and Uncertainty]{Uniform Bounds and Uncertainty for Asymptotics of Representations of $p$-adic $\GL_N$ 
}

\author{Rahul Dalal}
\author{Mathilde Gerbelli-Gauthier}
\author{Simon Marshall}

\date{\today}

\begin{document}

\begin{abstract}
We prove two results on the growth of dimensions of fixed vectors of representations $\pi$ of $p$-adic $\GL_N$ under principal congruence subgroups: First, a uniform bound on the growth of fixed vectors in terms of the GK-dimension $\pi$, which we extend to a uniform bound on the Harish-Chandra--Howe coefficients. Second, for $\pi$ unitary, a quantitative relationship between the GK-dimension of $\pi$ and the rate of decay of its matrix coefficients. These results are independent of one another and proved in the framework of the Langlands and Zelevinsky classifications. 
\end{abstract}

\maketitle



\tableofcontents

\section{Introduction}

Let $F$ be a $p$-adic field with ring of integers $\mc O$, residue field $k$ of size~$q$, and uniformizer $\varpi$. We study smooth, irreducible representations of ~$\GL_N(F)$. Our focus is on analytic invariants which arise in global applications, especially in the study of densities of non-tempered automorphic representations \cite{SX91, HK93, Mar14, MS19, Blo19, DGG22, EP22, AB24, DEP}
. The invariants are:
\begin{itemize}
    \item[(i)] the rate of growth of dimensions spaces of fixed vectors of $\pi$ under principal congruence subgroups of increasing level,
    \item[(ii)] the singularity of the character $\Theta_\pi$ of $\pi$ near the identity, encoded via the Harish-Chandra--Howe coefficients and the subsequent notions of wavefront set and GK-dimension,
    \item[(iii)] the rates of decay of matrix coefficients of $\pi$. 
\end{itemize}


The goal of this paper is twofold: to give uniform bounds on these invariants, with an initial focus on the rate of growth of fixed vectors (Theorem \ref{thm:mainfixedvector}), and to establish quantitative relationships between them, the most novel of which relates the GK-dimension and the rate of decay of matrix coefficients (Theorem \ref{thm:mainuncertainty}). 

\subsection{Analytic invariants}
We now describe the invariants in more detail.

\subsubsection{Fixed-Vector Growth}

For a positive integer $\ell$, let $K_\ell$ be the $\ell^{\text{th}}$ principal congruence subgroup of $\GL_N$:
\[
K_\ell := K_\ell^N := \{g \in \GL_N(F) \,:\, g \equiv \m I_N \pmod{\varpi^\ell} \},
\]
where $N$ is omitted when it is clear from context. Our first invariant is the function $\ell \mapsto \dim(\pi^{K_\ell})$, which is known to grow polynomially in $q^{\ell}$, see \eqref{eq:GKexpansion}. In the global context, this invariant is key to relating counts of automorphic forms (important in analytic applications and more easily accessible via the trace formula) to counts of automorphic representations (more natural from an arithmetic point of view).

We call the smallest integer $\ell(\pi)$ such that $\pi^{K_{\ell(\pi)}} \neq 0$ the \emph{level} of $\pi$. 


\subsubsection{Local Character Expansion}\label{sec HCH intro}
Let $\mf g$ be the Lie algebra of $\GL_N$, and let $\mc N := \mc N(N)$ be the set of nilpotent coadjoint orbits in the dual Lie algebra $\mf  g^*$.  (Note that henceforth we shall identify $\mf g^*$ with $\mf g$ using a symmetric nondegenerate invariant bilinear form.)  To $O \in \mc N$ corresponds a $\GL_N(F)$-invariant orbital integral measure $\mu_O \in C^\infty_c(\mf g)^*$, which, together with its Fourier transform $\wh{\mu}_O$, can be be interpreted as a distribution on $\mf g$. 

Let $\pi$ be a smooth, admissible representation of $\GL_N(F)$ with character $\Theta_\pi$. The following is a special case of a result known for general reductive $p$-adic groups:

\begin{thm}[Local Character Expansion, \cite{HC99admissible,Howe74Fourier}]\label{thm:HCH}
For a suitable normalization of the $\mu_O$, there are constants $c_O(\pi)$ for $O \in \mc N$ such that for small enough $X \in \mf g$, 
the trace character of $\pi$ admits the expansion
\[
\Theta_\pi(\exp X) = \sum_{O \in \mc N} c_O(\pi) \wh \mu_O(X).
\]
\end{thm}

The set of maximal orbits (under the closure ordering) satisfying that $c_O(\pi)\neq 0$
is known as the \emph{wavefront set} of $\pi$.  For $O$ in the wavefront set, $c_O(\pi)$ is equal to the dimension of a corresponding space of degenerate Whittaker models following \cite[II.2 Proposition]{MW87} in residual characteristic $p \neq 2$ and \cite{Var14} for $p=2$. The wavefront sets of local representations are connected to the vanishing of the various generalized Fourier coefficients of global automorphic representations (see \cite{JL24wavefront} for an overview).

For $\GL_N$, the wavefront set is determined by \cite{MW87}, although for general reductive groups its computation is still an open question and an active subject  of research, see e.g. \cite{tsai2023wave, ciubotaru2025wavefront}.  
Even more subtle, and not fully understood even for $\GL_N$, are the coefficients associated to non-maximal orbits, see e.g. \cite{Mur91GL3, Mur03local, HV24GLn}. These coefficients are fundamental objects in harmonic analysis as ``spectral-side'' analogues of Shalika germs for orbital integrals---e.g, \cite{Mur96}. 

The character expansion determines the asymptotic rate of growth of fixed vectors via the following notion:

\begin{dfn}
The Gelfand-Kirillov dimension, or GK-dimension of $\pi$ is
\[
d_{GK}(\pi) = \f12 \max \{ \dim O : c_O(\pi) \neq 0\}. 
\]
Let $D := D_N$ be the set of possible GK-dimensions $\{1/2\dim O : O \in \mc N(N)\}$.
\end{dfn}

It is well-known\footnote{See the proof of Lemma \ref{lem:homogenity} for a reconstruction of the argument.} that 
\begin{equation} \label{eq:GKexpansion}
\dim(\pi^{K_{\ell}}) \asymp q^{d_{GK}(\pi)\ell},
\end{equation}
where we use $\asymp$ to denote that both functions are bounded by a multiple of the other for $\ell$ large enough. The implied constants a priori depend on $\pi$. 


\subsubsection{Matrix Coefficient Decay}
Finally, we let $G$ be a reductive group over a real or $p$-adic field, we consider $\pi$ an admissible representation of $G$ on a Hilbert space, and introduce the $L^p$-integrability exponent $p(\pi)$ of $\pi$.  For any $v \in \pi$ and $w \in \pi^\vee$, we can define the matrix coefficient on $G$ given by
\[
m(v,w) := \langle gv, w \rangle.
\]
Next, if $\pi$ has central character $\om$, then for any matrix coefficient $m$ of $\pi$, $m\om^{-1}$ is defined on the central quotient $G/Z_G$. Therefore, the integrals
\[
\|m\|_{q, \om} = \lf(\int_{G/Z_G} |m\om^{-1}|^q \ri)^{1/q}
\]
are well-defined (but possibly infinite). We say that $m \in L^q_\om(G)$ if $\|m\|_{q, \om} < \infty$.

\begin{dfn}\label{def:ppi}
Given an admissible Hilbert space representation $\pi$ of a real or $p$-adic reductive group $G$ with central character $\om$, let
\[
p(\pi) := \inf\{q \geq 2 \,:\, \text{ the $K$-finite matrix coefficients of }\pi \text{ are in } L^q_\om(G)\}. 
\]
\end{dfn}

Representations such that $p(\pi) = 2$ are \emph{tempered}. The na\"ive Ramanujan conjecture would give that $p(\pi) = 2$ for any local component of a cuspidal global automorphic representation. This is still expected to hold for global $\GL_N$, but \cite{HPS79} famously showed that it does not in general, and counterexamples also exist for groups such that $G(F) = {\rm GL}_N(F)$, e.g. \cite{Ro90}. Therefore it is crucial in analytic applications to understand how frequently these $p(\pi)$ can be large---see the Sarnak-Xue conjecture introduced in \cite{SX91}. 

We note that for $\GL_N$, tempered representations are generic (i.e. have maximal wavefront set), but that the converse does not hold. 

\subsection{Results}
We work with the explicit combinatorial classifications of smooth irreducible representations of $\GL_N(F)$ in terms of supercuspidals, which we review in Section \ref{sec:background}. In Section \ref{sec:fixedvector}, we prove a uniform upper bound on the growth of dimensions of $K_\ell$-fixed vectors in terms of the GK-dimension, making the asymptotic \eqref{eq:GKexpansion} uniform:
\begin{thm}[Corollary \ref{cor:bodyfixedvector}]\label{thm:mainfixedvector}
For all $\eps > 0$ there is a uniform constant $C_\eps := C_{\eps, N, F}$ such that for any smooth irreducible representation $\pi$ of $\GL_N(F)$, we have
\[
\dim(\pi^{K_\ell}) \leq C_\eps q^{\ell (d_{GK}(\pi) + \eps)}.
\]
\end{thm}
The proof is given in terms of the ``Zelevinsky classification'', for which the intermediate building blocks are so-called Speh representations. For supercuspidal representations, a uniform upper bound was established in \cite[Lem A.1]{MS19} using the Whittaker model. The main contribution here is Theorem \ref{thm:pre speh bound} extending the supercuspidal bound to Speh representations using the generalized Whittaker models of \cite{MW87}. 

In Section \ref{sec:HCH}, we use a homogeneity property of the distributions $\wh \mu_O$ appearing in the HCH-expansion of Theorem \ref{thm:HCH} to turn Theorem \ref{thm:mainfixedvector} into a uniform bound on the coefficients~$c_O(\pi)$ in terms of the GK-dimension and level: 

\begin{cor}[Corollary \ref{cor:maincoefficientsalt}]\label{cor:maincoefficientsaltintro}
Let $\eps > 0$. Then there is a uniform constant $C_\eps := C_{\eps, N, F}$ such that for all smooth irreducible representations $\pi$ of $\GL_N(F)$ and all $O \in \mc N(N)$, we have
\[
|c_O(\pi)| \leq C_\eps q^{\ell(\pi)(d_{GK}(\pi) - 1/2\dim O + \eps)}.
\]
\end{cor}
In view of the results of \cite[\S 6]{Mur03local}, we believe that these bounds are sharp.

These results are an instance of the rather general principle that the Zelevinsky classification is well-adapted to computing analytic properties of representations near the identity. On the other hand, the decay of matrix coefficients can be computed from the ``Langlands classification'', built from discrete Steinberg representations. In the latter half of the paper, we study of the relationship between the GK-dimension and decay of matrix coefficients by playing the two classifications against each other via Aubert-Zelevinsky duality. 

Starting in Section \ref{sec:Atype}, we focus on unitarizable and Arthur-type representations. These are most relevant in global applications, and constitute a special case of the ``ladder representations'' of \cite{LM14} for which the Aubert-Zelevinsky dual can be computed combinatorially.  We first review the classification of the unitary dual and use it to prove a variant of Theorem \ref{thm:mainfixedvector}, Corollary \ref{cor:genboundrel}, which is more suited to global analytic applications. 

Finally, in Section \ref{sec:uncertainty}, we give a method to compute the $p(\pi)$ of Definition \ref{def:ppi} using the notion of ``exponents'' from \cite{HC73, casselman1995introduction}. In the case of Arthur-type representations this reduces to an especially nice formula, Theorem \ref{cor:mcdecaysimple}, which we use to quantitatively relate the non-genericity and non-temperedness of $\pi$.

\begin{thm}[Theorem \ref{thm:atypemctogk}]\label{thm:mainuncertainty}
Define \[g(\pi)=1 - \frac{2d_{GK}(\pi)}{N(N-1)}.\] Then for all Arthur-type irreducible representations $\pi$ of $\GL_N(F)$:
\[
g(\pi) \leq 1 - \f2{p(\pi)} \leq g(\pi)^{1/2}. 
\]
For general unitarizable $\pi$, the same identity holds with the right-hand side replaced by $g(\pi)^{1/2} + 2/N$.
\end{thm}

When $p(\pi)=2$, this is simply the statement that Arthur-type representations are tempered if and only if they are generic: our new contribution is to measure the relationship between $p(\pi)$ and $d_{GK}(\pi)$ for more degenerate representations. As a corollary, we give bounds on fixed-vector growth and HCH-coefficients in terms of $p(\pi)$.

\begin{rmk}

It is tempting to call Theorem \ref{thm:mainuncertainty} an uncertainty principle.  Indeed, the fixed-vector growth and HCH-coefficients describe the behavior of $\pi$ ``around the identity" while matrix coefficient decay relates to the behavior of $\pi$ ``at infinity". The numerical identity we establish describes a phenomenon likely known to experts: for unitarizable representations of $\GL_N$, more singular analytic behavior of the character near the identity, characterized by a larger GK dimension, corresponds to more well-behaved (i.e. closer to square-integrable) behavior of the matrix coefficients at infinity, and vice versa.
\end{rmk}

\subsection{Generalizations}
The focus of this paper is on $\GL_N$, but we expect both uniform bounds on fixed vectors and an uncertainty principle to hold for complex-valued representations of general reductive $p$-adic groups. For the uniform upper bounds, a crucial input is a bound on the support of certain matrix coefficients of supercuspidal representations, which is conjectured in \cite[\S 4]{FLM11} to hold generally and is known in several cases.

Beyond this, our proofs rely heavily on the explicit combinatorics of the Langlands and Zelevinsky classification for $\GL_N$ as detailed in \cite{LM14, LM16}. It is therefore unlikely that our methods would carry through verbatim. For classical groups, we mention recent progress on the local Jiang's conjecture, which aims to describe the wavefront set of representations inside of an Arthur packet \cite{LS25} in terms of the attached $A$-parameter. This, combined with known computations of the $p(\pi)$ for $A$-packets of classical groups in terms of the Arthur $SL_2$ \cite{DEP, EGGG}, could provide an approach to prove an uncertainty principle for classical $A$-packets.

As the most general possible conjecture for fixed vectors:
\begin{conj}
Let $G$ be a reductive group over $p$-adic field $F$, $x$ a point in its Bruhat-Tits building, and $K_\ell := K_{G, x, \ell}$ the $\ell$th Moy-Prasad subgroup at $x$. Then for all $\eps > 0$, there is a uniform constant $C := C_{\eps, G, F, x}$ such that for any smooth irreducible representation $\pi$ of $G(F)$:
\[
\dim(\pi^{K_\ell}) \leq C_\eps q^{\ell (d_{GK}(\pi) + \eps) }.
\]
\end{conj}
Note again that uniformity is the key difficulty here since the non-uniform variant holds by the HCH expansion.

\subsubsection*{Acknowledgements}
We thank Marie-France Vigneras and Yiannis Sakellaridis for help conceptually understanding the area, Alberto Minguez and Johannes Droschl for teaching us useful facts about the representation theory of $p$-adic $\GL_N$, Alex Cowan for giving us the tools to make the computations that led to the formulation of the uncertainty result, and Karol Koziol for jokingly suggesting that we refer to said result as an ``uncertainty principle" (though he should not be held responsible). 

RD was supported by Principal Investigator project PAT4832423 of the Austrian Science Fund (FWF) while working on this project. MGG was supported by an NSERC Discovery Grant, and SM was supported by NSF grant DMS-1902173.

\subsubsection*{Notational Conventions}
Denote by $[a_1^{(r_1)}, \dotsc, a_n^{(r_n)}]$ the unordered multiset containing $r_i$ copies of each $a_i$. Partitions are represented as multisets. 

Parentheses $(a_1^{(r_1)}, \dotsc, a_n^{(r_n)})$ instead denote the ordered $(r_1 + \cdots + r_n)$-tuple/list where the $a_i$ term is repeated $r_i$ times. 

The disjoint union symbols $\sqcup$  and $\bigsqcup_{i=1}^n$ represent either disjoint union of multisets or concatenation of ordered lists (ordered by the $i$ index for the big square cup). 

\section{Background}\label{sec:background}
\subsection{Nilpotent orbits in $\GL_N$}

Nilpotent orbits in $\GL_N$ are parameterized by the set $\mc P(N)$ of partitions of $N$ via the Jacobson-Morozov theorem, see e.g. \cite{CM93nilpotent}. 
In this parameterization, the natural closure ordering on orbits corresponds to the partial ordering on $\mc P(N)$ where $P_1 \preceq P_2$ if $P_1$ refines $P_2$. For example, the trivial, minimal, and regular nilpotent orbit correspond respectively to 
\[ 
P_{0}=[1^{(N)}], \quad P_{\mathrm{min}}=[2,1^{(N-1)}], \quad P_{\text{max}}=[N]. 
\] 
The set $\mc P(N)$ is equipped with an involution sending $P$ to its dual partition 
\begin{equation*} \label{eq: dual partition}
\hat{P} :=  [\#{\text{ parts of }P\text{  of size}\geq j}]_j. 
\end{equation*}
One derives a dimension formula for nilpotent orbits. If $\hat{P} = [\hat{N}_1,...,\hat{N}_r]$, then 
\begin{equation}\label{eq orbitdimension}
\dim O_P = N^2 - \sum_{i=1}^r \hat{N}^2_i.
\end{equation}

\subsection{Bernstein-Zelevinsky Classifications}
Bernstein--Zelevinsky give two classifications of the irreducible representations of $\GL_N$. One is the Langlands classification, in which irreducibles are realized as quotients of parabolic inductions of tempered/Steinberg representations. The other is the Zelevinsky classification, in which irreducible representations appear as subrepresentations of parabolic inductions of anti-tempered/Speh representations. 
We recall the combinatorics of these classifications: the original references are \cite{BZ77,Zel80}, but we follow the summary in \cite[\S2]{LM16}.

\subsubsection{Multisegments}
Let $N_i$ for $1 \leq i \leq r$ be such that $\sum_{i=1}^r N_i = N$. Given representations $\pi_i$ of $\GL_{N_i}$, define the normalized parabolic induction 
\[
\pi_1 \times \cdots \times \pi_r := \mc I_{\GL_{N_1}(F) \times \cdots \times \GL_{N_r}(F)}^{\GL_N(F)} (\pi_1 \boxtimes \cdots \boxtimes \pi_r)
\]
For a representation $\pi$ of $\GL_N$, we denote by $|\cdot|^x \pi$ the twist of $\pi$ by the $x\th$ power of the norm-of-determinant character. 

\begin{dfn}
Let $\rho$ be a  supercuspidal representation of $\GL_N$, and $a,b \in \R$ such that $b-a \in \Z_{\geq 0}$. The corresponding \emph{segment} is the set:
\[
\langle a, b \rangle_{\rho} := \{|\cdot|^a\rho, |\cdot|^{a+1} \rho, \dotsc, |\cdot|^b \rho\}.
\]
Note that the triple $(a,b,\rho)$ is not uniquely determined by the segment.

Two segments $\langle a, b \rangle_\rho$ and $\langle a',b' \rangle_{\rho'}$ are said to be \emph{linked} if their union is also a segment and neither is contained in the other. We say that $\langle a, b \rangle_\rho$ \emph{precedes} $\langle a',b' \rangle_{\rho'}$ if they are linked and $|\cdot|^{a'}\rho' = |\cdot|^{a + j}\rho$ for some integer $j > 0$. 
\end{dfn}

\begin{note}
There is always a unique way to represent a segment as $\langle a, b \rangle_\rho$ for $\rho$ unitary. If $\rho,\rho'$ are unitary, then $\langle a, b \rangle_\rho$ precedes $\langle a',b' \rangle_{\rho'}$ if and only if:
\begin{enumerate}
   \item $\rho = \rho'$,
   \item $a' - a \in \Z$,
   \item $a < a', b < b', a' \leq b+1$. 
\end{enumerate}  
\end{note}

\begin{dfn}
A \emph{multisegment} is a multiset of segments, denoted $[\langle a_i,b_i \rangle_{\rho_i}]_i$. We say that $[\langle a_i,b_i \rangle_{\rho_i}]_i$ is \emph{ordered} if for all $i < j$, $\langle a_i,b_i \rangle_{\rho_i}$ does not precede $\langle a_j,b_j \rangle_{\rho_j}$. 
\end{dfn}

Both the Langlands and Zelevinsky classifications of representations of $\GL_N(F)$ are given in terms of multisegments. 

\subsubsection{Langlands Classification}

\begin{thm}[Langlands classification]\label{thm:langlandsclassification}
Every smooth, irreducible representation~$\pi$ of $\GL_N(F)$ corresponds to a unique multisegment $\check M(\pi) = [\langle a_i,b_i \rangle_{\rho_i}]_{i=1}^r$ as follows:
\begin{enumerate}
    \item For each segment $m_i = \langle a_i,b_i \rangle_{\rho_i}$, there is a unique irreducible quotient:
    \[
    |\cdot|^{a_i} \rho_i \times \cdots \times |\cdot|^{b_i} \rho_i \onto L(m_i),
    \]
    which is necessarily square-integrable modulo the center. 
    \item $\pi$ is the unique irreducible quotient:
    \[
    L(m_1) \times \cdots \times L(m_r) \onto L(\check M) = \pi
    \]   
    whenever the $m_i$ are indexed so that $(m_1, \dotsc ,m_r)$ is ordered. 
    \item If none of the $m_i$ are linked, then $L(m_1) \times \cdots \times L(m_r)$ is irreducible. 
    \item If all the $\rho_i$ are unitary, $\pi$ is tempered modulo the center if and only if the $a_i + b_i$ are all equal. 
\end{enumerate}
\end{thm}

Representations constructed in (1) of the above theorem are known as (determinant twist of) \emph{generalized Steinberg representations}. Accordingly, we will denote
\[
L(\langle a, b \rangle_{\rho}) =: |\cdot|^{(a+b)/2} \St(b-a + 1, \rho).
\]

Lastly, we introduce an invariant from which we will compute the rate of decay of matrix coefficients of $\pi$ in \S \ref{sec:matrixcoefficientdecay}.
\begin{dfn} \label{def: mainexponent}
Given a multisegment $\check M = [\langle a_i,b_i \rangle_{\rho_i}]_{i=1}^r$, with  supercuspidal representations $\rho_i$ of~$\GL_{N_i}(F)$, define its \emph{character} $\Xi(\check M)$ to be the multiset of cardinality~$N$ given by
\[
\Xi(\check M) := \lf[\lf(\f{a_1 + b_1}2\ri)^{(N_1(b_1 - a_1 + 1))}, \dotsc, \lf(\f{a_r + b_r}2\ri)^{(N_r(b_r - a_r +1))}\ri],
\]
and set
\[
\Xi(\pi) := \Xi(\check M(\pi)). 
\]

\end{dfn}

\subsubsection{Zelevinsky Classification}

\begin{thm}[Zelevinsky classification]
Every smooth irreducible representation $\pi$ of $\GL_N(F)$ corresponds to a unique multisegment $M(\pi) = [\langle a_i,b_i \rangle_{\rho_i}]_{i=1}^r$ as follows:
\begin{enumerate}
    \item For each segment $m_i = \langle a_i,b_i \rangle_{\rho_i}$, there is a unique irreducible subrepresentation:
    \[
    Z(m_i) \into |\cdot|^{a_i} \rho_i \times \cdots \times |\cdot|^{b_i} \rho_i.
    \]
    \item $\pi$ is the unique irreducible subrepresentation:
    \[
    \pi = Z(M) \into Z(m_1) \times \cdots \times Z(m_r)
    \] 
    whenever the $m_i$ are indexed so that $(m_1, \dotsc, m_r)$ is ordered. 
    \item If none of the $m_i$ are linked, then $Z(m_1) \times \cdots \times Z(m_r)$ is irreducible.
\end{enumerate}
\end{thm}

Similarly to the case of the Langlands classification, we will use the notation 
\begin{equation} \label{eq: Speh def}
Z(\langle a, b \rangle_{\rho}) =: |\cdot|^{(a+b)/2} \Speh(b-a + 1, \rho) 
\end{equation}
and call the corresponding representation a (determinant twist of a) \emph{Speh representation}.\\

The following invariant will determine the wavefront set, and ultimately the GK-dimension, of a representation $\pi$ in the Zelevinsky classification:

\begin{dfn} \label{def: partition}
 Given a multisegment $M = [\langle a_i,b_i \rangle_{\rho_i}]_{i=1}^r$, with supercuspidals $\rho_i$ of $\GL_{N_i}(F)$, we define its \emph{partition}
\[
P(M) := [(a_1-b_1+1)^{(N_1)}, \dots, (a_r-b_r+1)^{(N_r)}]
\]
and write
\[
P(\pi) := P(M(\pi)). 
\] 
\end{dfn}



\subsubsection{Duality}
The space of smooth irreducible representations of a reductive $p$-adic group $G$ is equipped with the \emph{Aubert-Zelevinsky} involution $\pi \mapsto \check{\pi}$. For $G=GL_N$, this involution has the property:

\begin{thm}[{\cite[\S A.5]{LM16}}]\label{thm:dualsegment}
Let $\pi$ be a smooth irreducible representation of $\GL_N(F)$. Then $M(\pi) = \check M(\check \pi)$.  
\end{thm}

Moeglin--Waldspurger \cite{MW86} give an algorithm to compute $M(\check \pi)$ from $M(\pi)$ and vice versa.

\subsection{Degenerate Whittaker Models}
We next introduce degenerate Whittaker models, which arise in \cite{MW87} in the proof  of Theorem \ref{thm:wavefront}. They will play a central role in our uniform upper bounds. 

Let $U \subseteq \GL_N(F)$ be the upper-triangular unipotent matrices. Given a partition $P = [d_1, \dotsc, d_r]$ of $N$, let 
\[
S = \{d_1, d_1 + d_2, \cdots, d_1 + \cdots + d_{r-1}\} \subseteq \{1, \dotsc, N-1\}
\]
and define a character $\psi_P$ of $U$ by
\[
\psi_P : g \mapsto \psi_F\lf(\sum_{i \notin S} g_{i,i+1} \ri),
\]
where the $g_{i,i+1}$ are the above-the-diagonal entries of $g$ and $\psi_F$ is a fixed additive character of $F$ with conductor $\mc O$.

\begin{dfn}
Given a smooth representation $\pi$ of $\GL_N(F)$, a \emph{$P$-Whittaker model} for $\pi$ is an embedding 
\begin{equation} \label{eq: def whittaker model}
\pi \into W_P := \{f \in C^\infty(\GL_N(F),\BC) : f(ug) = \psi_P(u) f(g) \, \forall u \in U\},
\end{equation}
where $W_P$ is a $\GL_N(F)$-representation under right-translation. The $W_P$ for different choices of $\psi_F$ and indexings of the $d_i$ can be identified. When $P \neq [N]$, a $P$-Whittaker model for $\pi$ is said to be \emph{degenerate}. A representation is \emph{generic} if it admits an $[N]$-Whittaker model, which we simply call a \emph{Whittaker model}. 
\end{dfn}

\begin{thm}[{\cite[\S II.2]{MW87} using \cite{Zel80}}]\label{thm:degenwhittaker}
Let $\pi$ be a smooth irreducible representation of $\GL_N(F)$ and let $P(\pi)$ be as in Definition \ref{def: partition}. Then $\pi$ admits a unique~$\hat P(\pi)$-Whittaker model up to scalar. Furthermore, $\pi$ admits no $Q$-Whittaker model for $Q$ that doesn't refine $\hat P(\pi)$. 
\end{thm}

\begin{proof}
Note that \cite[\S II.2]{MW87} simply summarizes the parts of \cite{Zel80} explaining which degenerate Whittaker models $\pi$ admits and that this argument works in all residue characteristics. 
\end{proof}

\subsubsection{Restrictions of Whittaker Functions}
We recall a property of degenerate Whittaker models of Speh representations:

\begin{lem}\label{lem:restriction}
Let $\pi$ be a Speh representation with degenerate $\hat P(\pi)$-Whittaker model $\mc W(\pi)$. Consider the subgroup $\GL_{N-1}(F) \subset \GL_N(F)$ embedded in the upper-left corner. Then the map 
\[ 
\mc W(\pi) \to C^\infty(\GL_{N-1}(F)) 
\] 
given by restriction of functions is injective.
\end{lem}

\begin{proof}
Let $P_N \subset \GL_N(F)$ be the mirabolic subgroup, defined by the property that the bottom row of $p \in P_N$ is $(0, \dotsc, 0 ,1)$. Observe that 
\begin{equation} \label{eq pn decomp}
   P_N = U\cdot \GL_{N-1}. 
\end{equation}

Let $f \in \mc W(\pi)$. It follows  directly from \eqref{eq pn decomp} and the equivariance property \eqref{eq: def whittaker model} that if $f$ vanishes identically on $\GL_{N-1}(F)$, then it also does so on $P_N$.

This reduces the Lemma to showing injectivity of the restriction map
\[ 
\iota: \mc W(\pi) \to C^\infty(P_N). 
\] 
By Remark 3.6 in \cite{Zel80},  the restriction of representations $\pi|_{P_N}$ is irreducible, a property which extends to its Whittaker model 
$\mc W(\pi) \simeq \pi$. Then since $\ker \iota$ is a $P_N$-subrepresentation of $\mc W(\pi)$, injectivity is equivalent to $\ker \iota \neq \mc W(\pi)$. 

To show this, let $f \in \mc W(\pi)$ be a non-zero function and let $f(g) \neq 0$. Let $a \in \GL_N$ be such that $ga^{-1} \in P_N$. Then $(af)(ga^{-1}) \neq 0$ so $af \not\in \ker \iota$ and $\ker \iota \neq \mc W(\pi)$.
\end{proof}


\subsection{Wavefront Sets} 
\begin{dfn}
The wavefront set $WF(\pi)$ of a smooth irreducible representation $\pi$ of $\GL_N(F)$ is the set of maximal elements of $\{O \in \mc N(N) : c_O(\pi) \neq 0\}$ under the closure ordering; equivalently, it is the set of maximal elements of $\{P \in \mc P(N) : c_{O_P}(\pi) \neq 0\}$ under the refinement ordering. 
\end{dfn}

Moeglin--Waldspurger compute $WF(\pi)$ from the Zelevinsky classification by relating them to the degenerate Whittaker models that $\pi$ admits:

\begin{thm}[{\cite{MW87}, \cite{Var14}}]\label{thm:wavefront}
Let $\pi$ be a smooth irreducible representation of $\GL_N(F)$. Then 
\[ 
WF(\pi) = \left\{\wh{ P(\pi)}\right\}. 
\] 
\end{thm}

\begin{proof}
This follows from Theorem \ref{thm:degenwhittaker} together with the main result of \cite{MW87} (resp. \cite{Var14} in residual characteristic $2$). 
\end{proof}

Together with \eqref{eq orbitdimension}, we can then compute GK-dimensions:

\begin{cor}\label{cor:gkdimform}
If $\pi$ is a smooth irreducible representation of $\GL_N(F)$ with $M(\pi) = [\langle a_i,b_i \rangle_{\rho_i}]_{i=1}^r$ for $\rho_i$ on $\GL_{N_i}(F)$, then
\[
d_{GK}(\pi) = \f12 \lf(N^2 - \sum_{i=1}^r N_i (b_i - a_i + 1)^2 \ri). 
\]
\end{cor}

\begin{ex}\label{ex: Speh invariants}
If $d \mid N$, the Speh representation 
$\Sp(d,\rho) = Z(\langle \frac{1-d}{2}, \frac{d-1}{2} \rangle_\rho)$ has wavefront set and GK-dimension \[
WF(\Sp(d,\rho)) =O_{[(N/d)^{(d)}]}, \quad d_{GK}(\Sp(d,\rho)) = 1/2N(N-d).
\]  
\end{ex}

\section{Fixed Vector Bounds}\label{sec:fixedvector}

\subsection{Supercuspidals and Inductions}
We first recall a uniform bound for fixed vectors in supercuspidal representations proved in \cite{MS19}:
\begin{lem}[{\cite[Lem A.1]{MS19}}]\label{lem:supercuspidal bound}
Fix $N$ and $\eps > 0$. Then there is a uniform constant $C_\eps = C_{\eps,N}$ such that if $\pi$ is a supercuspidal of $\GL_N(F)$, then 
\begin{equation*} \label{eq: uniform upper bound sc}
\dim(\pi^{K_\ell}) \leq C_\epsilon q^{\ell \left(\frac{N(N-1)}{2} + \eps \right)}. 
\end{equation*}
\end{lem}

\begin{proof}
Sketching the argument in \cite{MS19}, the main result of \cite{Lap19} gives that for any supercuspidal $\pi$, all functions in $\mc W(\pi)^{K_\ell}$ are supported on some uniform set $S_\ell$ depending only on $\ell$. Then, Lemma \ref{lem:restriction} bounds the dimension of $\pi^{K_\ell}$ by the dimension of the space of $K_\ell$-invariant Whittaker functions supported on $\GL_{N-1} \cap S_\ell$ which can be explicitly computed. 
\end{proof}
Note that the exponent $N(N-1)/2$ is the GK-dimension of the supercuspidal (and hence generic) representation $\pi$. 

For later use, we also recall a bound on fixed vectors in parabolic inductions:
\begin{lem}\label{lem:inductionbound}
For all $N$, there are uniform $C_1, C_2$ such that for all 
\[
\pi = \rho_1 \times \cdots \times \rho_r
\]
with each $\rho_i$ a smooth representation of $\GL_{N_i}(F)$ and $\sum_{i=1}^r N_i = N$:
\[
C_1 \dim(\pi^{K_\ell}) \leq q^{\f12(N^2 - \sum_i N_i^2)} \prod_{i=1}^r \dim(\rho_i^{K_\ell}) \leq C_2 \dim(\pi^{K_\ell}). 
\]
\end{lem}

\begin{proof}
This is standard; see e.g. the proof of \cite[Lem 5.2]{MS19}. 
\end{proof}

\subsection{Speh Representations}
We now move to the intermediate and technically most complex step: bounding fixed vectors in Speh representations
.

\begin{thm}\label{thm:pre speh bound}
Let $N,k > 0$. There is a uniform constant $C := C_{N,k}$ such that for all supercuspidal representations $\rho$ of $\GL_N$ we have
\[
\dim(\Speh(k, \rho)^{K_\ell}) \leq C q^{\f12 \ell k(k-1)N(N-1)} \lf(\dim(\rho^{K_\ell})\ri)^k.
\]
\end{thm}

\begin{proof}
By induction, it suffices to find a uniform $C := C_{N,k}$ such that:
\[
\dim(\Speh(k, \rho)^{K_\ell}) \leq C q^{\ell kN(N-1)} \dim(\Speh(k-1, \rho)^{K_\ell}) \dim(\rho^{K_\ell}). 
\] 
Let $\pi = \Speh(k, \rho)$ and $\pi_0 = \Speh(k-1,\rho)$. Define the following subgroups of $\GL_{Nk}$:
\begin{itemize}
    \item $\GL_{Nk-1}$, $\GL_{N(k-1)}$ in the upper-left corner,
    \item $\GL_N$ in the lower-right corner,
    \item $\GL_{N-1}$ in the upper-left corner of $\GL_N$,
    \item $P$ the $(N(k-1), N-1)$-upper parabolic of $\GL_{Nk-1}$ and $V$ its unipotent radical. 
\end{itemize} 
By Theorem \ref{thm:degenwhittaker}, $\pi$ has a unique $[N^{(k)}]$-Whittaker model $\mc W(\pi)$; likewise $\pi_0$ has a unique $[N^{(k-1)}]$-Whittaker model $\mc W(\pi_0)$. Lastly, let $\mc W(\rho)$ be the Whittaker model of $\rho$. 

Let $\Theta$ be the space $\mc W(\pi)^{K^{Nk}_\ell}|_{\GL_{Nk-1}}$ of functions on $\GL_{Nk-1}$. By Lemma \ref{lem:restriction}, $\dim \pi^{K_\ell^{Nk}} \leq \dim \Theta$. Our goal is to bound $\dim \Theta$, which we will do by restriction to $P$.

By \cite[(3a)]{LM20}, there is unramified character $\chi$ of $\GL_{N(k-1)} \times \GL_N$ such that restriction of functions gives a containment:
\begin{equation} 
\label{eq: containment of Whittaker models}
\mc W(\pi)|_{\GL_{N(k-1)} \times \GL_N} \subseteq \chi \otimes (\mc W(\pi_0) \boxtimes \mc W(\rho)). 
\end{equation}
Let $\Psi$ be the space $\mc W(\rho)^{K_\ell^N}\mid_{GL_{N-1}}$ given by restricting functions. Let $\Phi$ be the space of functions on $P$ that are left-invariant under $V$ and whose restriction to $\GL_{N(k-1)}\times \GL_{N-1}$ lies in $\chi \otimes \mc W(\pi_0)^{K_\ell^{N(k-1)}}\otimes \Psi$. It follows from \eqref{eq: containment of Whittaker models} that if $f \in \mc W(\pi)^{K_\ell^{Nk}}$ then $f|_P \in \Phi$. 

 If we let $\gamma_i$ be a set of double coset representatives for 
\[ 
(K^{Nk-1} \cap P) \dom K^{Nk-1}/K^{Nk-1}_\ell,  
\]
then from the Iwasawa decomposition for $\GL_{Nk-1}$ we have that 
\[
\GL_{Nk-1} = \bigcup_i P\gamma_iK^{Nk-1}_\ell.
\]
If $f \in \Theta$, then the function $p\mapsto f(p\gamma_i)$ lies in $\Phi$ for all $i$. It follows that 
\[ 
\dim \Theta \leq \#\lf((K^{Nk-1} \cap P)\dom K^{Nk-1}/K^{Nk-1}_\ell \ri) \dim \Phi. 
\] 
Next, 
\[
\#(K^{Nk-1} \cap P)\dom K^{Nk-1}/K^{Nk-1}_\ell = \#P(\mc O/\varpi^\ell) \bs \GL_{Nk-1}(\mc O/\varpi^\ell) \leq C q^{\ell(N-1)N(k-1)}
\]
for some constant $C$, by a dimension count and the formula for the order of a general linear group over a finite field. Computing
\[ 
\dim \Phi = \dim \mc W(\pi_0)^{K^{N(k-1)}_\ell} \dim \Psi = \dim \pi_0^{K^{N(k-1)}_\ell} \dim \rho^{K^N_\ell}
\] 
then completes the proof.
\end{proof}

Substituting in Lemma \ref{lem:supercuspidal bound}:

\begin{cor}\label{cor:speh bound}
Let $N,k > 0$. Then for any $\eps > 0$, there is a uniform constant $C_\eps$ such that for all supercuspidal representations $\rho$ of $\GL_N$, 
\[
\dim(\Speh(k, \rho)^{K_\ell}) \leq C_\eps q^{ \ell (\f12 k^2N(N-1) + \eps)}.
\]
\end{cor}

It is also useful to sometimes use the bound
\begin{equation}\label{eq:alt speh bound}
\dim(\Speh(k, \rho)^{K_\ell}) \leq C_\eps q^{\ell (\f12 (k^2-1)N(N-1) + \eps)} \dim(\rho^{K_\ell})    
\end{equation}
in the same notation as Theorem \ref{thm:pre speh bound} and Corollary \ref{cor:speh bound}. 



\subsection{General Representations}\label{ssec:fixedvectorbound}
Combining the Speh bound of Corollary \ref{cor:speh bound} with the Zelevinsky classification gives the general bound:

\begin{cor}\label{cor:bodyfixedvector}
Let $N > 0$ and fix $N_i, k_i$ for $1 \leq i \leq r$ such that $\sum_i N_i k_i = N$. Then for any $\eps > 0$, there is a uniform constant $C_\eps$ such that for all smooth admissible irreducible representations $\pi$ with $M(\pi) = (\langle a_i, b_i \rangle_{ \rho_i})_{i=1}^r$ with $b_i - a_i + 1 = k_i$ and $\rho_i$ a supercuspidal representation of $\GL_{N_i}(F)$, 
\[
\dim(\pi^{K_\ell}) \leq C_\eps q^{\f12  \ell (N^2 - \sum_{i=1}^r N_i k_i^2 + \eps) }.
\]
\end{cor}

\begin{proof}
Let $x_i = (a_i + b_i)/2$. Then $\pi$ is a subrepresentation of the parabolic induction
\[
\sigma := |\cdot|^{x_i} \Speh(k_i, \rho_i) \times \cdots \times |\cdot|^{x_r} \Speh(k_r, \rho_r).
\]
Therefore, for some constant $C$,
\[
\dim(\pi^{K_\ell}) \leq \dim(\sigma^{K_\ell}) \leq Cq^{\f12 (N^2 - \sum_{i=1}^r N_i^2 k_i^2)} \prod_{i=1}^r \dim((|\cdot|^{x_i} \Speh(k_i, \rho_i))^{K_\ell}) 
\]
by Lemma \ref{lem:inductionbound}. Finally, by Corollary \ref{cor:speh bound}, there is $C_\eps$ such that
\begin{equation}\label{eq:in gen bound}
\dim((|\cdot|^{x_i} \Speh(k_i, \rho_i))^{K_\ell}) = \dim(\Speh(k_i, \rho_i)^{K_\ell}) \leq C_\eps q^{ \f12 \ell(N_i^2 k_i^2 - N_i k_i^2 + \eps)},
\end{equation}
from which the result follows. 
\end{proof}

In the notation of Corollary \ref{cor:bodyfixedvector}, Corollary \ref{cor:gkdimform} gives
\[
d_{GK}(\pi) = N^2 - \sum_{i=1}^r N_i k_i^2
\]
which depends only on the $k_i$ and $N_i$ and is exactly the exponent on $q^\ell$ in Corollary \ref{cor:bodyfixedvector}. Maximizing constants over the finitely many choices of $N_i, k_i$ gives Theorem \ref{thm:mainfixedvector}. 

If we input \eqref{eq:alt speh bound} or Theorem \ref{thm:pre speh bound} instead of Corollary \ref{cor:speh bound} in the intermediate equation \eqref{eq:in gen bound}, we get alternate uniform bounds
\begin{align}\label{eq:alt gen bound}
\dim(\pi^{K_\ell}) &\leq C_\eps q^{\ell \lf(d_{GK}(\pi) - \sum_{i=1}^r d_{GK}(\rho_i) + \eps \ri)} \prod_{i=1}^r \dim(\rho_i^{K_\ell}) \\
\dim(\pi^{K_\ell}) &\leq C_\eps q^{\ell \lf(d_{GK}(\pi) - \sum_{i=1}^r k_i d_{GK}(\rho_i) + \eps \ri)} \prod_{i=1}^r \lf(\dim(\rho_i^{K_\ell})\ri)^{k_i} \label{eq:alt gen bound 2}
\end{align}
in the same notation as Corollary \ref{cor:bodyfixedvector}.

\section{Homogeneity and Harish-Chandra-Howe Coefficients}\label{sec:HCH}
We introduce a notion of scaling families of functions that will allow us to translate the upper bounds on fixed vectors into bounds on the Harish-Chandra-Howe coefficients of a representation. We use the notation of Section \ref{sec HCH intro}.
\begin{dfn}
Let $f$ be a function on $\mf g$. For each $\ell \in \Z$, define
\begin{equation} \label{eq family of functions}
f^{(\ell)}(X) := q^{\ell \dim \mf g}f(\varpi^{-\ell} X).
\end{equation}
For $\ell$ large enough such that the exponential map is a bijection from $\varpi^{\ell} \Supp(f)$, also define
\[
\varphi^{(\ell)}(\exp X) = f^{(\ell)}(X). 
\]
We call such sequences $f^{(\ell)}$ on $\mf g$ or $\varphi^{(\ell)}$ on $G$ \emph{scaling families}. 
\end{dfn}

\begin{ex}\label{ex:scalingindicators}
The notion of scaling family is a direct generalization of the prototypical example \[f^{(\ell)} = \mu(L_{\ell})^{-1}\mathbf{1}_{L_{\ell}},\] for $L_\ell = \varpi^\ell M_n(\mc O)$.  These functions satisfy the defining property \eqref{eq family of functions} for $\ell \geq 1$. The corresponding family on $G$ is $\varphi^{(\ell)} = \bar{\mathbf{1}}_{K_\ell}$, the normalized indicator functions of the subgroups $K_{\ell}$. 
\end{ex}

\subsection{Homogeneity bounds from Character Expansions}

The next proposition is a variant for scaling families of a result that is certainly known to the experts in the case of $\varphi^{(\ell)} = \bar{\1}_{K_\ell}$, in which case it implies the identity 
\[
\dim(\pi^{K_\ell}) \asymp q^{\ell d_{GK}(\pi)}.
\]

\begin{lem}\label{lem:homogenity}
Let $f^{(\ell)}$ be a scaling family. For any $O \in \mc N$ we have 
\[ 
\hat \mu_O(f^{(\ell)}) = \hat{\mu}_O(f)q^{\ell \cdot \dim O/2}.
\]
\end{lem}

\begin{proof}
This follows directly from the homogeneity of the distributions~$\mu_{O}$ established in \cite{HC99admissible}.  Let $\phi \in C^\infty_c(\fg)$ and recall that by definition $\hat{\mu}(\phi) := \mu(\hat{\phi})$.
The Fourier transform on $\fg$ is defined as \[ \hat{\phi}(Y) = \int_\fg \phi(X) \psi(B(X,Y)) dX,  \] where $\psi$ is a fixed character of $F$ and $B$ is a nondegenerate symmetric invariant bilinear form.  The usual property of the Fourier transform holds in our setup: \begin{equation} \label{eq Fourier scaling}
    \widehat{
\phi(aX)} = \frac{1}{|a|^{\dim \fg}}\hat{\phi}(X/a).
\end{equation}  
We apply this to $f^{(\ell)}(X) = q^{\ell\cdot \dim \fg} f(\varpi^{-\ell}X)$ and get $\wh{f^{(\ell)}}(X) = \hat{f}(\varpi^{\ell}X)$.

The result of \cite[\S 3.1]{HC99admissible} shows that, given a family of functions $\phi_t(X) := \phi(tX)$, the nilpotent orbital integrals satisfy 
\[ 
\mu_{ O}(\phi_t) 
= |t^{-1}|^{\dim O/2}\mu_{ O}(\phi).
\] 
By applying this to $\phi = \hat f$ and $t=\varpi^\ell$, we see that 
\[
\mu_{ O}(\wh{f^{(\ell)}}) = |\varpi^{-\ell}|^{\dim  O/2}\mu_ O(\hat f) = q^{\ell \cdot \dim O/2}\mu_{O}( \hat f).
\] 
It folows that for each orbit $O$ we have 
\begin{equation}\label{eq:homogeneity}
\hat{\mu}_O(f^{(\ell)}) = \hat{\mu}_O(f) q^{\ell \cdot \dim O/2},
\end{equation} 
as claimed.
\end{proof}

\begin{rmk}
Using the Harish-Chandra--Howe expansion and applying Lemma~\ref{lem:homogenity}, the scaling family $\varphi^{(\ell)} = \bar \1_{K_\ell}$ of Example \ref{ex:scalingindicators} gives that for each smooth irreducible representation $\pi$ of $\GL_N(F)$:
\[
\tr_\pi(\bar \1_{K_\ell}) = \dim(\pi^{K_\ell}) \asymp q^{\ell d_{GK}(\pi)},
\]
where the implied constants in the $\asymp$ are non-uniform depending on $\pi$. Here we used that~$\wh \mu_O(\bar \1_{K_\ell}) > 0$ since $\bar \1_{K_\ell}$ has positive Fourier transform.  

A similar asymptotic with non-uniform constants holds for arbitrary $|\tr_\pi(\varphi^{(\ell)})|$ provided that $\sum_{O \in WF(\pi)} \wh \mu_O(f) \neq 0$, where $\varphi^{(\ell)}$ is the scaling family associated to $f$ (the sum is a singleton in our $\GL_N$ case). 
\end{rmk}

\subsection{Homegeneity bounds from Fixed Vectors}
The fixed vector bound of Theorem \ref{thm:mainfixedvector} can be upgraded to a bound on traces against scaling families. In contrast with the above remark, note that our goal here is uniform constants. 
\begin{prop}\label{prop:scalingtrace}
Let $\varphi^{(\ell)}$ be a scaling family defined from $f : \mf g \to \C$. 
Let~$\pi$ be a smooth irreducible representation of $\GL_N(F)$. Then for any $\epsilon>0$, there is a constant $C_{\eps, f} := C_{\eps, N, f, F}$ independent of $\pi$ such that for large enough $\ell$, we have
\[
|\tr_\pi (\varphi^{(\ell)})| \leq C_{\eps, f} q^{\ell (d_{GK}(\pi) + \eps)}. 
\]
\end{prop}

\begin{proof}
Recall $L_\ell = \varpi^\ell M_n(\mc O)$. Since the $L_\ell$ form a basis at the identity for the topology on $\mf g$, there exists~$r$ such that for some fixed constants $c_i \in \C$, and $a_i \in \mf g$:
\[
f = \sum_i c_i \bar \1_{a_i+L_r} \implies f^{(\ell)} = \sum_i c_i \bar \1_{\varpi^\ell a_i + L_{r + \ell}}.
\]
Since for large $\ell$, $\exp(L_\ell) = K_\ell$, the Zassenhaus formula gives that $\exp(\varpi^\ell a_i + L_{r + \ell}) = \exp(\varpi^\ell a_i) K_{r+\ell}$ for $\ell$ larger than some threshold $\ell_0$ depending only on $a_i,N,F$. In total, for $\ell \geq \ell_0$
\[
\varphi^{(\ell)} = \sum_i c_i \bar \1_{\exp(\varpi^\ell a_i) K_{r+\ell}}. 
\]
Next, $\bar \1_{\exp(\varpi^\ell a_i) K_{r+\ell}}$ acts on $\pi$ by projecting to $\pi^{K_{r + \ell}}$ and then acting by $\exp(\varpi^\ell a_i)$. We can without loss of generality choose $\ell_0$ so that all the $\exp(\varpi^\ell a_i)$ are also contained in some compact subgroup and therefore act unitarily with respect to some Hermitian form. This gives the bound
\[
|\tr_\pi (\bar \1_{\exp(\varpi^\ell a_i) K_{r+\ell}})| \leq \dim \pi^{K_{r+\ell}}. 
\]
In total
\[
|\tr_\pi(\varphi^{(\ell)})| \leq \sum_i |c_i| |\tr_\pi (\bar \1_{\exp(\varpi^\ell a_i) K_{r+\ell}}) | \leq C_f \dim \pi^{K_{r+\ell}}
\]
for some constant $C_f$ depending only on $f$. The result then follows by Theorem \ref{thm:mainfixedvector}.
\end{proof}

\subsection{Harish-Chandra-Howe Coefficients}
Proposition \ref{prop:scalingtrace} and the homogeneity result of Lemma \ref{lem:homogenity} let us give a uniform bound on the coefficients $c_{O}(\pi)$ in terms of the level of $\pi$ and its GK-dimension. 


\begin{cor}\label{cor:maincoefficientsalt}
Let $\eps > 0$. Then there is a uniform constant $C_\eps := C_{\eps, N, F}$ such that for all smooth irreducible representations $\pi$ of $\GL_N(F)$ and all $O \in \mc N(N)$, we have
\[
|c_O(\pi)| \leq C_\eps q^{\ell(\pi)(d_{GK}(\pi) - 1/2\dim O + \eps)}.
\]
\end{cor}

\begin{proof}
There is a function $f_O$ on $\mf g$ such that for all $O' \in \mc N(N)$, $\wh \mu_{O'}(f_O) = \1_{O = O'}$ (see \cite[\S6.4]{Kot05} for the corresponding function for $\mu_O$ and take Fourier transforms). Define the corresponding scaling family $\varphi_O^{(\ell)}$ for all $\ell \geq \ell_0$, where $\ell_0$ is the threshold appearing in the proof of Proposition \ref{prop:scalingtrace} and depends only on $f_{O}$. 

By Proposition \ref{prop:scalingtrace}, for all $\eps> 0$, there is $C_{\eps, O}$ such that for any $\pi$ and any $\ell > \ell_0$
\[
|\tr_\pi(\varphi_O^{(\ell)})| \leq C_{\eps,O} q^{\ell (d_{GK}(\pi) + \eps)}. 
\]
On the other hand, the main result of \cite{DB02} gives that for $\ell(\pi) \geq \ell_0$, 
\[
\tr_\pi(\varphi_O^{(\ell(\pi) + 1)}) = \sum_{O' \in \mc N(N)} c_{O'}(\pi) \wh \mu_{O'}(f^{(\ell(\pi) + 1)}_O) = c_O(\pi) q^{1/2 (\ell(\pi) + 1) \dim O},
\]
where the last equality follows from the homogeneity property of Lemma \ref{lem:homogenity} and $\wh \mu_O(f_O) = 1$. 

Combining the two equations and maximizing the constants over $O$ we have 
\[
|c_O(\pi)| \leq q^{d_{GK}(\pi) - 1/2\dim O} C_\eps q^{\ell(\pi)(d_{GK}(\pi) - 1/2\dim O + \eps)} 
\]
whenever $\ell(\pi) \geq \ell_0$. We can without loss of generality change the constant to handle all $\pi$ by using an alternate bound by $q^{\ell_0(d_{GK}(\pi) - 1/2\dim O + \eps)}$ for small $\ell$. 
\end{proof}

\section{Arthur-type and Unitarizable Representations}\label{sec:Atype}


\subsection{Classification}
In global applications, we mostly care about unitarizable irreducible representations. These have a more refined classification with particularly nice properties which we now recall.

For $1 \leq i \leq r$, let $\rho_i$ be a supercuspidal on $\GL_{N_i}(F)$ and let $a_i, d_i$ be positive integers. We denote
\begin{equation}\label{eq:Lform}
\bigoplus_{i=1}^r |\cdot|^{x_i} \rho_i[a_i][d_i]:= L(\check M)
\end{equation}
where we recall that $L(\check M)$ denotes a Langlands quotient, and $\check M$ is the multisegment
\begin{equation} \label{eq:Arthur type as Langlands Quotient}
\bigsqcup_{i=1}^r \bigsqcup_{j=1}^{d_i} \lf[\lf\langle x_i + \f{d_i-2j +1}2+ \f{1-a_i}2, x_i + \f{d_i-2j +1}2 +\f{a_i-1}2 \ri\rangle_{\rho_i} \ri].
\end{equation}

\begin{dfn}\label{def:Atype}
The representation $\pi$ is \emph{Arthur-type} if 
\[ 
\pi = \bigoplus_{i=1}^r \rho_i[a_i][d_i],
\]
i.e. if all the $x_i$ in \eqref{eq:Lform} are $0$.
\end{dfn}

\begin{rmk}
This definition is motivated by Arthur's classification: if $\pi$ is Arthur-type 
and $\rho_i$ corresponds to the supercuspidal Langlands parameter $\varphi_i : W_F \to \GL_{N_i}(\C)$, where $W_F$ is the Weil group of $F$, then $\pi$ has Arthur parameter
\[
\psi_\pi  = \bigoplus_{i=1}^r \varphi_i \boxtimes [a_i] \boxtimes [d_i] : W_F \times \SL_2(\C) \times \SL_2(\C) \to \GL_N(\C) ,
\]
where $[d]$ represents the $d$-dimensional representation of $\SL_2(\C)$. 
\end{rmk}

\begin{thm}[{\cite{Tad86b}}]\label{thm:unitarizable}
A smooth irreducible representation $\pi$ of $\GL_N(F)$ is unitarizable if and only if it is of the form
\[
\pi = \bigoplus_{i=1}^r |\cdot|^{x_i} \rho_i[a_i][d_i] = \bigoplus_{i=1}^{r_1}\rho_i[a_i][d_i] \oplus \bigoplus_{i=r_1+1}^{r_2} ( |\cdot|^{y_i} \rho_i[a_i][d_i] \oplus |\cdot|^{-y_i} \rho_i[a_i][d_i]),
\]
where for each $i$, $0 < y_i < 1/2$. (In this case, it can be uniquely written in this form). 
\end{thm}

\subsubsection{Inducing Data}
Unitarizable $\pi$ have particularly nice inducing data:
\begin{prop}[{\cite{ber06}}]\label{prop:unitaryinduction}
Let $\pi_i$ be unitarizable representations of $\GL_{N_i}(F)$ for $1 \leq i \leq k$. Then $\pi_1 \times \cdots \times \pi_k$ is irreducible. 
\end{prop}

\begin{proof}
See also \cite[\S6.3]{LM16} for a simplified proof. 
\end{proof}

As a well-known related result:

\begin{cor}\label{cor:unitarizableisfullinduction}
Let
\[
\pi := \bigoplus_{i=1}^r |\cdot|^{x_i} \rho_i[a_i][d_i] := \bigoplus_{i=1}^{r_1}\rho_i[a_i][d_i] \oplus \bigoplus_{i=r_1 + 1}^{r_2} ( |\cdot|^{y_i} \rho_i[a_i][d_i] \oplus |\cdot|^{-y_i} \rho_i[a_i][d_i])
\]
be a unitarizable irreducible representation of $\GL_N(F)$ as in Theorem \ref{thm:unitarizable}. Then $\pi$ is equal to the full parabolic induction
\[
\pi = \bigtimes_{i=1}^r |\cdot|^{x_i} \rho_i[a_i][d_i].
\]
\end{cor}

\begin{proof}


For $1 \le i \le r$, we let $M_i$ be the multisegment corresponding to $|\cdot|^{x_i} \rho_i[a_i][d_i]$, and for $r_1 + 1 \le i \le r_2$, we let $M_i'$ correspond to $|\cdot|^{-y_i} \rho_i[a_i][d_i]$.  Let $r_1 + 1 \le i \le r_2$ be given. 

Because the segments in $M_i$ and $M_i'$ are pairwise unlinked, Theorem 2.6 of \cite{LM16} (and in particular the statement dual to part 6 of Proposition 2.5) implies that $L(M_i) \times L(M_i')$ is irreducible, and hence that $L(M_i + M_i') = L(M_i) \times L(M_i')$.  Theorem \ref{thm:unitarizable} then implies that $L(M_i) \times L(M_i')$ is unitarizable.  

It follows that the representations
\[
\bigtimes_{i = 1}^{r_1} L(M_i) \times \bigtimes_{i = r_1 + 1}^{r_2} (L(M_i) \times L(M_i'))
\]
and $\bigtimes_{i = 1}^r L(M_i)$ are both irreducible by Theorem \ref{prop:unitaryinduction} and \cite[Thm 1.9]{Zel80} respectively.  We may now apply Theorem 2.6 of \cite{LM16} (in particular the statement dual to part 5 of Proposition 2.5) inductively to deduce that $L(M_1 + \ldots + M_j) = L(M_1) \times \ldots \times L(M_j)$ for all $j$, as required.
\end{proof}


\subsection{Basic Invariants}

\subsubsection{Duality}
Aubert--Zelevinsky duals of unitarizable representations are easy to calculate: duality amounts to ``swapping the two copies of $
\SL_2$".

\begin{prop}\label{prop:Atypedual}
Let $\pi$ be a unitarizable irreducible representation of $\GL_N(F)$. Then $\check \pi$ is also unitarizable and 
\[
\pi = \bigoplus_{i=1}^r |\cdot|^{x_i} \rho_i[a_i][d_i] \iff \check \pi = \bigoplus_{i=1}^r |\cdot|^{x_i} \rho_i[d_i][a_i].
\]
\end{prop}

\begin{proof}
This is well-known and can be checked by the algorithm of \cite{MW86} using that $|x_i| < 1/2$---see, e.g, \cite[\S3.2]{LM14}.  
\end{proof}

\subsubsection{Arthur-$\SL_2$} For unitarizable $\pi$, asymptotic invariants are determined by the following coarsening of the defining data. 

\begin{dfn}
If $\pi$ is unitarizable with
\[
\pi = \bigoplus_{i=1}^r |\cdot|^{x_i} \rho_i[a_i][d_i]
\]
where $\rho_i$ is a supercuspidal of $\GL_{N_i}(F)$:
\begin{itemize}
\item
The \emph{Arthur-$\SL_2$} of $\pi$ is the partition
\[
A(\pi) := [d_1^{(N_1 a_1)}, \dotsc, d_r^{(N_r a_r)}].  
\]
If $\pi$ is of Arthur type, this is the partition of $N$ given by restricting the Arthur parameter $\psi_\pi$ to the second (Arthur) $\SL_2$.
\item The \emph{augmented Arthur-$\SL_2$} is the multiset of pairs 
\[
A^u(\pi) :=  [(x_1, d_1)^{(N_1 a_1)}, \dotsc, (x_r,d_r)^{(N_r a_r)}].
\]
\end{itemize}
Note that if $\pi$ is Arthur-type, the two notions encode the exact same information.
\end{dfn}

\subsubsection{GK-dimension}
The dependence on the Arthur-$\SL_2$ is most immediate for the wavefront set and GK-dimension:

\begin{cor}[{see also \cite[Thm 1.1]{LS25}}]\label{cor:ASL2toGKdim}
Let $\pi$ be a unitarizable irreducible representation of $\GL_N(F)$. Then $WF(\pi) = \{\widehat{A(\pi)}\}$. In particular, if $A(\pi) = [d_1, \dotsc, d_s]$, then
\[
d_{GK}(\pi) = \f12\lf(N^2 - \sum_{i=1}^s d_i^2\ri). 
\]
\end{cor}

\begin{proof}
We have
\[
WF(\pi) \overset{\text{Thm. }\ref{thm:wavefront}}=  \left\{\widehat {P(M(\pi))}\right\} \overset{\text{Thm. }\ref{thm:dualsegment}}= \left\{\widehat{P(\check M(\check \pi))}\right\}. 
\]
This reduces the problem to showing that $A(\pi) = P(\check M (\check \pi))$, which follows from Proposition \ref{prop:Atypedual} and Definition \ref{def: partition}.
\end{proof}

\subsection{A Variant Fixed-Vector Bound}
For global applications, it is convenient to have a fixed vector bound analogous to \eqref{eq:alt speh bound} in terms of the Langlands, rather than Zelevinsky, classification. 
We begin by proving a lower bound on fixed vectors in Steinberg representations:

\begin{lem}\label{lem:steinberglower}
Choose $N,k$. Then there are uniform constants $C_1, C_2$ such that for $\rho$ a supercuspidal on $\GL_N(F)$ and $q^\ell \geq C_1$:
\[
\dim(\St(k, \rho)^{K_\ell}) \geq C_2 q^{\f12 \ell N^2k(k-1)} \lf(\dim(\rho^{K_\ell})\ri)^k.
\]
\end{lem}

\begin{proof}
By \cite[\S 9.1]{Zel80}, the Zelevinsky data of $\St(k, \rho)$ is $[|\cdot|^{\f{k-1}2} \rho, \ldots, |\cdot|^{\f{-k+1}2} \rho]$, so that it is a subrepresentation of
\[
\St(k, \rho) \subseteq \sigma:=|\cdot|^{\f{k-1}2} \rho \times |\cdot|^{\f{k-3}2} \rho \times \cdots \times |\cdot|^{\f{-k+1}2} \rho. 
\]
Moreover, $\St(k, \rho)$ occurs with multiplicity 1 in $\sigma$ by \cite[Thm 7.1]{Zel80}.  By Lemma~\ref{lem:inductionbound}, there is a uniform constant $C$ such that
\[
\dim(\sigma^{K_\ell}) \geq Cq^{\f12\ell N^2k(k-1)} \lf(\dim(\rho^{K_\ell})\ri)^k.
\]
Consider the non-Steinberg irreducible subquotients $\pi$ of $\sigma$, which satisfy:
\[
\dim(\St(k, \rho)^{K_\ell}) = \dim( \sigma^{K_\ell}) - \sum_{\pi \neq \St(k, \rho)} \dim(\pi^{K_\ell}).
\]
We give an upper bound on this rightmost sum.

Let $\pi \neq \St(k, \rho)$ be one of the other subquotients. By uniqueness of supercuspidal supports, 
we have $M(\pi) = [m_1, \dotsc, m_r]$ with $r < k$ and
\[
\bigsqcup_{i=1}^r m_i = \lf \langle \f{1-k}2, \f{k-1}2 \ri\rangle_\rho. 
\]
In particular, Corollary \ref{cor:gkdimform} implies that $d_{GK}(\pi) < \f12 Nk(Nk-1)$, so that \eqref{eq:alt gen bound 2} gives
\[
\dim(\pi^{K_\ell}) \leq C'q^{\f12\ell (N^2k(k-1) - a)} \lf(\dim(\rho^{K_\ell})\ri)^k 
\]
for some $a > 0$ and uniform $C'$. Finally, $\pi$ decomposes into $2^{k-1}$ irreducible subquotients by \cite[Cor 2.3]{Zel80}. 
\end{proof}

Combining this with \eqref{eq:alt gen bound 2} gives the following bound for representations associated to generalized Arthur parameters: 


\begin{cor}\label{cor:genboundrel}
For any $\eps > 0$, there is a uniform constant $C_\eps := C_{\eps, N, F}$ such that the following holds.  If $\psi$ is a generalized\footnote{i.e, in Arthur's set $\Psi^+$} Arthur parameter of the form
\[
\psi = \bigoplus_{i=1}^r \varphi_i[d_i]: WD_F \times \SL_2(\C) \to \GL_N(\C),
\]
where $\varphi_i$ are $L$-parameters corresponding to generic, unitarizable irreducible representations $\sigma_i$, then the representation $\pi$ corresponding to $\psi$ is irreducible and unitary, and 
\[
\dim(\pi^{K_\ell}) \leq C_\eps q^{\ell (d_{GK}(\pi) - \sum_{i=1}^r d_{GK}(\sigma_i)+\eps)} \prod_{i=1}^r \dim(\sigma_i^{K_\ell}). 
\]
\end{cor}

\begin{proof}
By Corollary \ref{cor:ASL2toGKdim} and Theorem \ref{thm:degenwhittaker}, genericity gives that each $A(\sigma_i) = [1, \ldots, 1]$.  Combined with Theorem \ref{thm:unitarizable} and Corollary \ref{cor:unitarizableisfullinduction}, this implies that $\sigma_i$ is of the form
\[
\sigma_i = |\cdot|^{x_{i,1}} \St(a_{i,1}, \rho_{i,1}) \times \cdots \times |\cdot|^{x_{i,s_i}}\St(a_{i, s_i}, \rho_{i, s_i}).
\]
If $\eta_{i,j}$ is the Langlands parameter corresponding to $\rho_{i,j}$, this implies that
\[
\varphi_i = \bigoplus_{j=1}^{s_i} |\cdot|^{x_{i,j}} \eta_{i,j}[a_{i,j}].
\]
Let the set of distinct $x_{i,j}$'s be $X_1 > \ldots > X_t$, and for each $X_k$ let $\psi_k$ be the sum of the $\eta_{i,j}[a_{i,j}][d_i]$ with $X_k = x_{i,j}$, so that
\[
\psi = \bigoplus_{k=1}^t |\cdot|^{X_k} \psi_k.
\]
If $\psi_k$ corresponds to the irreducible representation $\pi_k$, then \cite[p. 27--28]{Mok} gives that
\[
\pi = |\cdot|^{X_1} \pi_1 \times \ldots \times |\cdot|^{X_k} \pi_k.
\]
Corollary \ref{cor:unitarizableisfullinduction} implies that each $\pi_i$ is the full induction of some subset of the $\rho_{i,j}[a_{i,j}][d_i]$, and hence that $\pi$ is the full induction of all of the $|\cdot|^{x_{i,j}} \rho_{i,j}[a_{i,j}][d_i]$ in some order.  Another application of Corollary \ref{cor:unitarizableisfullinduction} and \cite[Thm 1.9]{Zel80} gives that $\pi$ is irreducible and unitary, and the induction is independent of the ordering of $|\cdot|^{x_{i,j}} \rho_{i,j}[a_{i,j}][d_i]$.

We can then apply equation \eqref{eq:alt gen bound}, which gives:
\begin{align*}
\dim(\pi^{K_\ell}) &\leq C_\eps q^{\ell (d_{GK}(\pi) - \sum_{i=1}^r \sum_{j=1}^{s_i} a_{i, j} d_{GK}(\rho_{i,j})+\eps)} \prod_{i=1}^r \prod_{j=1}^{s_i} \lf( \dim(\rho_{i,j}^{K_\ell}) \ri)^{a_{i,j}} \\
&\leq C_\eps q^{\ell (d_{GK}(\pi) - \sum_{i=1}^r \sum_{j=1}^{s_i}  d_{GK}(\St(a_{i,j}, \rho_{i,j}))+\eps)} \prod_{i=1}^r \prod_{j=1}^{s_i} \dim(\St(a_{i,j}, \rho_{i,j})^{K_\ell})
\end{align*}
by Lemma \ref{lem:steinberglower} and Corollary \ref{cor:ASL2toGKdim} to compute GK-dimensions in the exponent. The result then follows using Lemma \ref{lem:inductionbound} to collapse over $j$, noting that each $\sigma_i$ is generic. 
\end{proof}
\begin{rmk}
The above result is motivated by the problem of counting non-tempered automorphic forms on unitary groups as the split level varies. If 
\[
\psi = \bigoplus_{i=1}^r \tau_i[d_i]
\]
is a global Arthur parameter, then the $\tau_i$ are cuspidal automorphic representations of $\GL_N$ so their local components $\tau_{i,v}$ are generic and unitary. In particular, at any prime $v$, the localization
\[
\psi_v = \bigoplus_{i=1}^r \tau_{i,v}[d_i]
\]
corresponds to an irreducible representation of the form in Corollary \ref{cor:genboundrel}. (The Ramanujan conjecture would further guarantee that the $\tau_i$ are tempered.) 

In \cite[Conj 9.6.3]{DGG22}, the first two authors make a conjecture about an optimal upper bound for these counts in terms of the global Arthur-$\SL_2$. Corollary~\ref{cor:genboundrel} combined with the arguments of \cite{DGG22} or \cite{MS19} is enough to prove this conjecture. It therefore implies an upper bound with the expected optimal exponent on the growth of counts of automorphic forms on unitary groups with fixed cohomological component at infinity
---see \cite[Thm 11.4.2]{DGG22} and the discussion afterwards. This will be explained in more detail in a future work. 
\end{rmk}

\section{Matrix Coefficient Decay and Uncertainty}\label{sec:uncertainty}

In this final section, we show that the matrix coefficient decay of a unitarizable representation $\pi$ is determined by its augmented Arthur-$\SL_2$, and use this result to compare it to the GK-dimension of $\pi$.

\subsection{Exponents}
Following Casselman \cite{casselman1995introduction}, the matrix coefficient decay of a representation is determined by a key set of invariants called \emph{exponents}.

\subsubsection{Definition and Properties}




\begin{dfn}[\cite{casselman1995introduction,HC73}]
Fix a minimal parabolic $P_0$ of $\GL_N$. Let $\pi$ be a smooth irreducible representation of $G = \GL_N(F)$. An \emph{exponent} $\om$ of $\pi$ is a central character of an irreducible subquotient of the normalized Jacquet functor $\mc R^G_M(\pi)$ associated to the Levi factor $M$ of a standard parabolic $P$. 
\end{dfn}

The character $\omega$ defined in this way is the central character of the Levi subgroup $M = \prod_j \GL_{N_j}(F)$ which we can decompose as
\[ 
\omega = \boxtimes_j\omega_j, \quad \omega_j = \phi_j|\det|^{\alpha_j} : Z_{\GL_{N_j}(F)} \to \C^\times
\]
for $\phi_j$ unitary and $\alpha_j \in \R$. Since $P_0$ determines an ordering of the blocks of $M$, we abuse notation and also denote by $\om$ the ordered $N$-tuple of numbers
\[ 
\om = \left(\alpha_j^{(N_j)}\right)_j. 
\]
This is normalized so that the unramified character on the torus determined by this list of numbers restricts on $Z_M$ to the ``unramified part'' of $\om$ as above. 

In order to compute with exponents, we introduce some combinatorial notation. 

\begin{dfn}\label{def: sigma_i}
Given a sequence $a = (a_1, \dotsc, a_n)$ of real numbers: 
\begin{itemize}
\item Let $\sigma_i(a)= \sum_{j\leq i} a_j$ (where out-of-bounds elements are treated as $0$ for indexing purposes). 
\item Given another sequence $a' = (a'_1, \dotsc, a'_n)$, we say that $a$ is dominated by $a'$, or $a \preceq a'$, if for all $i$, $\sigma_i(a) \leq \sigma_i(a')$. 
\end{itemize}
Given a multiset $\Xi$ of real numbers, we can abuse notation and treat it as a sequence in non-increasing order and similarly define the above notions. 
\end{dfn}

\subsubsection{Minimal Exponent}
When introducing the Langlands classification in Definition \ref{def: mainexponent}, we defined a so-called character $\Xi(\pi)$. It is readily computed for Arthur-type and unitarizable $\pi$.

\begin{lem}\label{lem:charactercomputation}
Let $\pi$ be an Arthur-type irreducible representation of $\GL_N(F)$ with Arthur-$\SL_2$
\[
A(\pi) = [d_1, \dotsc, d_s].
\]
Then $\pi$ has character
\[
\Xi(\pi) = \bigsqcup_{i=1}^s \lf[ \f{d_i-1}2, \f{d_i-3}2, \dotsc, \f{-d_i+1}2\ri]. 
\]
If $\pi$ is only unitarizable with augmented Arthur-$\SL_2$ 
\[
A^u(\pi) = [(d_1, x_1), \dotsc, (d_s, x_s)],
\]
then $\pi$ has character 
\[
\Xi(\pi) = \bigsqcup_{i=1}^s \lf[ x_i + \f{d_i-1}2, x_i + \f{d_i-3}2, \dotsc, x_i + \f{-d_i+1}2\ri]. 
\]
In particular, the character $\Xi(\pi)$ is determined by the augmented Arthur-$\SL_2$ of $\pi$.
\end{lem}

\begin{conv}
As convention, $\Xi(\pi)$ will also denote its rearrangement in non-increasing order. For unitarizable $\pi$, note that $-\Xi(\pi)$ is the same list ordered non-decreasingly.  
\end{conv}



The next three results establish that $-\Xi(\pi)$ 
is the minimal exponent of $\pi$ in the ordering $\prec$. We will next use this to show that $\Xi(\pi)$ determines the rate of decay of matrix coefficients of unitarizable representations. 



\begin{lem}\label{lem:simpleexponents}
Let $\rho$ be a supercuspidal representation of $\GL_m(F)$. Then the ordered list
\begin{equation}\label{eq:intermediate exp}
\lf(\lf(\f{-d + 1}2\ri)^{(ma)}, \dotsc, \lf(\f{d-1}2\ri)^{(ma)} \ri)
\end{equation}
is an exponent of $\pi : = \rho[a][d]$.
\end{lem}

\begin{proof}
Let $N = mad$, and let $M$ (resp. $L$) be the Levi factor of the standard parabolic for $\GL_N(F)$ determined by the ordered list $(m^{(ad)})$ (resp. $((ma)^{(d)})$). We will realize \eqref{eq:intermediate exp} as an exponent with respect to $L$. 

Theorem 7 of \cite{LM14} computes all the exponents of $\pi$ with respect to $M$ in terms of certain labelings $\phi$ of a ladder graph $\mc E(\pi)$. In particular, in the notation of \cite[\S3.3]{LM14}, the ``costandard'' labeling given by
\[
\begin{tikzpicture}[every node/.style={font=\tiny}]
\node[label = $a$] (A0) at (0,0)   {$\bullet$};
\node[label = $a-1$] (A1) at (1.5,0)   {$\bullet$};
\node[label = $2$] (A2) at (3,0)   {$\bullet$};
\node[label = $1$] (A3) at (4.5,0)   {$\bullet$};

\node[label = $2a$] (B0) at (1.5,1.5) {$\bullet$};
\node[label = $2a-1$] (B1) at (3,1.5) {$\bullet$};
\node[label = $a+2$] (B2) at (4.5,1.5) {$\bullet$};
\node[label = $a+1$] (B3) at (6,1.5){$\bullet$};

\node[label = $da$] (D0) at (3,3) {$\bullet$};
\node[label = $da-1$] (D1) at (4.5,3) {$\bullet$};
\node[label = $(d-1)a + 2$] (D2) at (6,3) {$\bullet$};
\node[label = right:$(d-1)a + 1$] (D3) at (7.5,3){$\bullet$};

\draw[->] (A0) -- (B0);
\draw[dotted,-]  (B0) -- (D0);

\draw[->] (A1) -- (B1);
\draw[dotted,-]  (B1) -- (D1);

\draw[->] (A2) -- (B2);
\draw[dotted,-] (B2) -- (D2);

\draw[->] (A3) -- (B3);
\draw[dotted,-]  (B3) -- (D3);

\draw[->]            (A0) -- (A1);
\draw[dotted,-]     (A1) -- (A2);
\draw[->,]           (A2)  -- (A3);
\draw[->]            (B0) -- (B1);
\draw[dotted,-]     (B1) -- (B2);
\draw[->,]           (B2)  -- (B3);
\draw[->]            (D0) -- (D1);
\draw[dotted,-]     (D1) -- (D2);
\draw[->,]           (D2)  -- (D3);
\end{tikzpicture}
\]
 gives the exponent
\[
E = \bigsqcup_{i=1}^d \lf( \lf(\f{a-1}2 + \f{-d - 1 + 2i}2\ri)^{(m)}, \dotsc, \lf(\f{-a + 1}2 + \f{-d - 1 + 2i}2\ri)^{(m)} \ri),
\]
concatenated without reordering. 

As explained \cite[\S 4.4]{casselman1995introduction}, restriction to $Z_L \subset Z_M$ of the central character of $M$ determined by $E$ gives rise to an exponent associated to $L$.
With our normalizations, this corresponds to averaging $E$ over the blocks of $L$, which gives 
\[
E_{\text{averaged}} = \bigsqcup_{i=1}^d \lf( \lf(\f{-d - 1 + 2i}2\ri)^{(ma)} \ri),
\]
as claimed.
\end{proof}

\begin{prop} \label{prop: -xi is an exp}
Let $\pi$ be a unitarizable irreducible representation of $\GL_N(F)$. Then $-\Xi(\pi)$ is an exponent of $\pi$. 
\end{prop}

\begin{proof}
Let
\[
\pi = \bigoplus_{i=1}^r |\cdot|^{x_i} \rho_i[a_i][d_i],
\]
with $\rho_i$ on $\GL_{N_i}(F)$. Let $M$ be the Levi factor of the standard parabolic $P_M$ corresponding to ordered list $(N_1 a_1 d_1, \dotsc, N_r a_r d_r)$ and $L$ the Levi corresponding to $P_L$ from $((N_1 a_1)^{(d_1)}, \dotsc, (N_r a_r)^{(d_r)})$. 
Let $\sigma$ be the representation
\[
\sigma = \boxtimes_{i=1}^r |\cdot|^{x_i} \rho_i[a_i][d_i],
\]
of $M$. By Corollary  \ref{cor:unitarizableisfullinduction}, $\pi = \mc I_M^{\GL_N(F)} \sigma$, so it suffices to construct $-\Xi(\pi)$ as an exponent of this full induction, for which we use Bernstein's Geometric Lemma.  

By Lemma \ref{lem:simpleexponents} applied to each factor of $M$, $\sigma$ has an exponent 
\[
E := \bigsqcup_{i=1}^r \bigsqcup_{j=1}^{d_i} \lf(\lf(\f{-d_i + 1 + 2j}2 + x_i\ri)^{(N_ia_i)} \ri)
\]
concatenated without reordering. Note that by Lemma \ref{lem:charactercomputation}, $-\Xi(\pi)$ is obtained from $E$ by reordering the entries so that they are non-decreasing. 

We freely use notions from \cite[\S 2]{BZ77}. To apply Bernstein's Geometric Lemma \cite[Lem 2.12]{BZ77}, we must exhibit an element $w$ of the Weyl group $W$ of $\GL_N$, and a Levi subgroup $L'$ such that $w$ belongs to a certain subset $W^{M,L'}$ and $w(E) = -\Xi(\pi)$. It will then follow that $-\Xi(\pi)$ appears as an exponent associated to $\mc R^G_{L'} \mc I_M^G \sigma$. 

The district entries of $E$ within each block of $M$ (indexed by $i$) are non-decreasing, so we can pick a $w$ such that $w(E) = -\Xi(\pi)$ with the property that it reorders the blocks of $N$ while preserving the relative ordering of the entries of $E$ in each block of $M$.  We define $P_{L'} = w(P_L)$ with Levi factor $L'$ -- this is the Levi factor for which we claim that $-\Xi(\pi)$ is an exponent. Since $w^{-1}(P_{L'})=P_L$, none of the blocks of $L'$ are broken up by applying $w^{-1}$ -- they are simply permuted. A fortiori, their interior order is preserved. This ``preservation of the inner ordering of the blocks" is precisely the criterion for belonging to $W^{M,L'}$ in \cite[\S 2.11]{BZ77}, so $w\in W^{M,L'}$.

Therefore, by \cite[Lem. 2.12]{BZ77}, the functor $\mc I^{L'}_{L'} \circ w \circ \mc R^M_L$ appears as a composition factor of $\mc R^{G}_{L'}\mc I_M^G$. Since $E$ is an exponent coming from $\mc R^M_L \sigma$ and $w(E) = -\Xi(\pi)$, we conclude. 
\end{proof}


\begin{lem} \label{lem: dominant exponent}
Let $\pi$ be a smooth irreducible representation of $\GL_N(F)$. Then every exponent of $\pi$ dominates $w_l(\Xi(\pi))$, where  $w_l$ is the longest element in the Weyl group (i.e, reversing the order of sequences). 
\end{lem}

\begin{rmk}
$w_l(\Xi(\pi)) = -\Xi(\pi)$ when $\pi$ is unitarizable. 
\end{rmk}

\begin{proof}
We write $\pi = L([\langle a_i,b_i\rangle_{\rho_i}]_i)$ in the Langlands classification. Lemma 6.2.1 in \cite{DEP} gives the stronger statement that all the exponents of
\[
\sigma = \bigtimes_i [\langle a_i,b_i\rangle_{\rho_i}]_i
\]
dominate $w_l(\Xi(\pi))$. 
\end{proof}

\subsection{Matrix Coefficient Decay}\label{sec:matrixcoefficientdecay}
We are now ready to show that for $\pi$ unitarizable, the character $\Xi(\pi)$, and therefore the augmented Arthur-$\SL_2$, determine the matrix coefficient decay $p(\pi)$ introduced in Definition \ref{def:ppi}.


\begin{prop}\label{prop:matrixdecay}
Let $\pi$ be a unitarizable irreducible representation of $\GL_N(F)$. Then
\[
\f 2{p(\pi)} =  1 - \max_{1 \leq i \leq N-1} \f{2 \sigma_i(\Xi(\pi))}{i(N-i)}.
\]
\end{prop}

\begin{proof}
We prove this in the framework of \cite[\S 4]{casselman1995introduction} and apply a variant (which goes through mutatis mutandis) of Theorem 4.4.6 therein, where $p=2$ is replaced by $p>2$. 

We recall the bijection between standard parabolic subgroups of $\GL_N$ and subsets $\Theta \subset \Delta$ of the simple roots.  Let $A$ be the diagonal subgroup of $\GL_N$.  Given $\Theta \subset \Delta$, we let $A_\Theta \subset A$ be the common kernel of all $\alpha \in \Theta$, $M_\Theta$ be the centralizer of $A_\Theta$ in $\GL_N$, and $P_\Theta$ be the standard parabolic with Levi $M_\Theta$.  We also define
\[
A_{\Theta}^-=\{x\in A_{\Theta} : |\alpha(x)|\leq 1 \text{ for all } \alpha \in \Delta \setminus \Theta\}, 
\] 
and let $A^- = A_\emptyset^-$.  Theorem 4.4.6 of \cite{casselman1995introduction} then states that the matrix coefficients of $\pi$ belong to $L^{p+\epsilon}$ for all $\epsilon>0$ if and only if for any $\Theta \subset \Delta$ and any exponent $\chi$ associated to~$P_{\Theta}$ and viewed as a character of $A_{\Theta}$, we have 
\begin{equation} \label{eq: Casselman thm}
\left|\chi\delta^{\frac{1}{2}-\frac{1}{p}}(x)\right| \leq 1 \quad \text{for all } x\in A_\Theta^{-}.
\end{equation}
Note that this differs from Casselman's statement by $\delta^{1/2}$ since we are working with normalized parabolic induction. 

We next introduce some notation to distinguish between exponents viewed as ordered $N$-tuples and as characters.  If $E$ is an exponent of $\pi$ associated to the Levi $M_\Theta$, and viewed as an ordered $N$-tuple, we let $\chi_E$ be the corresponding character of $A_\Theta$.  (Note that $\chi_E$ is not uniquely determined by $E$, but $|\chi_E|$ is.)  We let $\widetilde{\chi}_E$ denote the unramified positive character of $A$ corresponding to the $N$-tuple $E$, which has the property that $\widetilde{\chi}_E = |\chi_E|$ on $A_\Theta$.

If $E$ and $E'$ are exponents of $\pi$, the definition of $\prec$ in Definition \ref{def: sigma_i} implies that
\[
E' \preceq E \iff \widetilde{\chi}_E(x) \leq \widetilde{\chi}_{E'}(x) \text{ for all } x \in A^-. 
\] 
We therefore conclude from Lemma \ref{lem: dominant exponent} that all exponents $E$ of $\pi$ satisfy 
\[  
\widetilde{\chi}_E(x) \leq \widetilde{\chi}_{-\Xi(\pi)}(x) \text{ for all } x \in A^-.
\] 
\medskip

Let $p'$ be defined by the condition
\begin{equation}\label{eq: pprime}
\f 2{p'} =  1 - \max_{1 \leq i \leq N-1} \f{2 \sigma_i(\Xi(\pi))}{i(N-i)}.
\end{equation}
We first show that $p(\pi) \le p'$, which means showing that \eqref{eq: Casselman thm} holds with $p = p'$.  Let $E$ be an exponent of $\pi$ associated to the Levi $M_\Theta$, and let $\chi_E$ and $\widetilde{\chi}_E$ be as above.  To show that $\left|\chi_E \delta^{\frac{1}{2}-\frac{1}{p'}}(x)\right| \leq 1$ on $A_\Theta^{-}$, it suffices to show that $\widetilde{\chi}_{-\Xi(\pi)} \delta^{\frac{1}{2}-\frac{1}{p'}}(x) \leq 1$ on $A^{-}$.  Moreover, it suffices to check this at the elements
\[
x_{i} = (\overbrace{\varpi,...,\varpi}^{i},1,...,1) \in A^-, \quad 1 \leq i \leq N-1
\] 
corresponding to the fundamental weights, as these generate the semigroup $A^- / A(\mathcal{O}) Z$. With the notation $\sigma_i$ introduced in Definition \ref{def: sigma_i}, we have 
\begin{equation}\label{eq: logchar}
\log_q \widetilde{\chi}_{-\Xi(\pi)} \delta^{\frac{1}{2}-\frac{1}{p'}}(x_i) = \sigma_i(\Xi(\pi))-\left(\frac{1}{2}-\frac1{p'} \right)i(N-i).
\end{equation}
Rearranging, the definition of $p'$ implies that the right hand side is nonpositive, as required.
\medskip

To show that $p(\pi) \ge p'$, we will show that \eqref{eq: Casselman thm} doesn't hold when $p < p'$ and $E = -\Xi(\pi)$.  If $-\Xi(\pi)$ corresponds to the Levi $M_\Theta$, it suffices to show that there is $x_i \in A_\Theta^-$ such that $| \chi_{-\Xi(\pi)} \delta^{\frac{1}{2}-\frac{1}{p'}}(x_i)| = 1$.  By \eqref{eq: logchar}, this will follow if the maximum in \eqref{eq: pprime} is realized at an $i$ such that $x_i \in A_\Theta^-$.

We do this using Lemma \ref{lem:indexofmaximum1} below.  To apply it, write $\Xi(\pi)$ (considered as an ordered list) as 
\[
(d_1^{(a_1)}, \dotsc, d_r^{(a_r)}, 0^{(b)}, -d_r^{(a_r)}, \dotsc, -d_1^{(a_1)})
\] 
with $d_i$ distinct.  If $M_\Sigma$ is the Levi corresponding to the partition $(a_1,...,a_r,b,a_r,...,a_1)$, then because $\Xi(\pi)$ is constant on the blocks of $M_\Theta$, we must have $M_\Theta \subset M_\Sigma$ and $A_\Sigma \subset A_\Theta$.  It therefore suffices to show that the maximum in \eqref{eq: pprime} is realized for $i$ such that $x_i \in A_\Sigma^-$, or equivalently (because we may assume without loss of generality that $i \le N/2$), of the form $a_1 + \ldots + a_j$ for some $j$, and this is established by Lemma \ref{lem:indexofmaximum1}.

\end{proof}

\begin{lem}\label{lem:indexofmaximum1}
Define an ordered list of length $n$ and indices $s_j$ by:
\[
\Xi = (d_1^{(a_1)}, \dotsc, d_r^{(a_r)}, 0^{(b)}, -d_r^{(a_r)}, \dotsc, -d_1^{(a_1)}), 
\qquad s_j = \sum_{i=1}^j a_i,
\]
where $d_1 > \cdots > d_r$. Then the maximum value 
\[ 
\max_{i \le N/2} \f{\sigma_i(\Xi)}{i(N-i)} 
\] 
can only be achieved at $i=s_j$ for some $j$.
\end{lem}

\begin{proof}
Let $\rho = (N-1, N-3, \dotsc, -N+3, -N + 1)$.  Any maximizing index must be at most $s_r$, and if it is not equal to $s_r$ we may write it as $s_j + x$ where $j < r$ is as large as possible and $0 \leq x < a_{j+1}$. Assume the statement is false so that $x \neq 0$. If $y$ and $z$ are the averages of $\rho_i$ over $s_j + 1 \le i \le s_j + x$ and $s_j + 1 \le i \le s_{j+1}$ respectively, then we have $y > z > 0$ since $\rho$ is decreasing and $s_{j+1} \le N/2$, and
\[
Q_{s_j + x} = \f{\sigma_{s_j}(\Xi) + x d_{j+1}}{\sigma_{s_j}(\rho) + xy}, \quad Q_{s_{j+1}} = \f{\sigma_{s_j}(\Xi) + a_{j+1} d_{j+1}}{\sigma_{s_j}(\rho) + a_{j+1} z}.
\]
Since
\[
\f{\sigma_{s_j}(\Xi) + x d_{j+1}}{\sigma_{s_j}(\rho) + xy} = Q_{s_j + x} \geq Q_{s_j} = \f{\sigma_{s_j}(\Xi)}{\sigma_{s_j}(\rho)}, 
\]
we know that 
\[
\alpha \mapsto \f{\sigma_{s_j}(\Xi) + \alpha d_{j+1}}{\sigma_{s_j}(\rho) + \alpha y}
\]
is non-decreasing in $\alpha > 0$. Combined with $y > z$, this gives
\[
Q_{s_{j+1}} = \f{\sigma_{s_j}(\Xi) + a_{j+1} d_{j+1}}{\sigma_{s_j}(\rho) + a_{j+1} z} > \f{\sigma_{s_j}(\Xi) + a_{j+1} d_{j+1}}{\sigma_{s_j}(\rho) + a_{j+1} y} \geq \f{\sigma_{s_j}(\Xi) + x d_{j+1}}{\sigma_{s_j}(\rho) + x y} = Q_{s_j + x}
\]
which is a contradiction. 
\end{proof}

We have the following simplification of Proposition \ref{prop:matrixdecay} for Arthur-type representations:

\begin{cor}\label{cor:mcdecaysimple}
Let $\pi$ be an Arthur-type irreducible representation of $\GL_N(F)$ such that 
\[
A(\pi) = (d_1^{(a_1)}, \dotsc, d_r^{(a_r)})
\]
with $d_1 > \cdots > d_r \geq 1$ and $a_i > 0$. Then
\[
\f2{p(\pi)} = 1 - \f{d_1 - 1}{N - a_1}. 
\]
\end{cor}

This follows from Proposition \ref{prop:matrixdecay}, and the following combinatorial lemma:

\begin{lem}\label{lem:indexofmaximum2}

Let $A$ be a sequence of the form $(d_1^{(a_1)}, \dotsc, d_r^{(a_r)})$ with $d_1 > \cdots > d_r \geq 1$ and $a_i > 0$.  Let $\Xi$ be the sequence defined as the non-increasing rearrangement of
\[
\bigsqcup_{i=1}^r \lf[ \left( \f{d_i-1}2 \right)^{(a_i)}, \left( \f{d_i-3}2 \right)^{(a_i)}, \dotsc, \left( \f{1-d_i}2 \right)^{(a_i)} \ri],
\]
and let $N = \sum_{i=1}^r d_i a_i$ be the length of $\Xi$.  Then
\[
\max_i \f{2\sigma_i(\Xi)}{i(N-i)} = \frac{d_1 - 1}{N - a_1}.
\]

\end{lem}

\begin{proof}

It is equivalent to show that the desired maximum occurs at $i = a_1$.  We prove this by induction on $d_1$.  The case $d_1 = 1$ is trivial, and $d_1 = 2$ follows immediately from Lemma \ref{lem:indexofmaximum1}.  We may therefore assume that $d_1 \ge 3$.

We write $2\Xi$ as
\[
2\Xi = ((d_1 - 1)^{(m_{1})}, (d_1 - 2)^{(m_{2})}, \dotsc, (1 - d_1)^{(m_{2d_1 - 1} )} ), 
\]
possibly with some $m_j = 0$, and let $s_k = \sum_{j= 1}^k m_j$.  We define $A'$ to be the sequence obtained from $A$ by replacing $d_1^{(a_1)}$ with $(d_1-2)^{(a_1)}$ and taking the non-increasing rearrangement, and let $\Xi'$ be obtained from $A'$ in the same way as $\Xi$ is from $A$.  It follows that
\[
2\Xi' = ((d_1 - 2)^{(m_{2})}, \dotsc, (2 - d_1)^{(m_{2d_1 - 2} )} ).
\]
We let $N' = N - 2m_1$ be the length of $\Xi'$, and define $s'_k = \sum_{j= 2}^{k+1} m_j$ for $k \ge 1$, so that $s'_k$ are associated to $\Xi'$ in the same way as $s_k$ are to $\Xi$.

By Lemma \ref{lem:indexofmaximum1}, the maximum is realized at $i = s_k$ for some $1 \le k \le d_1 - 1$, and we may assume without loss of generality that $s_k > m_1$.  We wish to show that
\[
\f{2\sigma_{s_k}(\Xi)}{s_k(N-s_k)} \le \frac{d_1 - 1}{N - a_1}.
\]
Breaking $\sigma_{s_k}(\Xi)$ into its first $m_1$ and last $s_k - m_1$ terms gives
\begin{align*}
\f{2\sigma_{s_k}(\Xi)}{s_k(N-s_k)} & = \frac{ (d_1-1)m_1 + \sum_{j=2}^k (d_1 - j) m_j }{  (N - m_1)m_1 + (N - s_k - m_1)(s_k - m_1) } \\
& = \frac{ (d_1-1)m_1 + 2\sigma_{s'_{k-1}}(\Xi') }{  (N - m_1)m_1 + (N' - s'_{k-1} )s'_{k-1} }.
\end{align*}
By the inequality
\begin{equation} \label{eq: dream function addition}
a,c \geq 0, \quad b,d > 0, \quad \f ab \geq \f cd \implies \f ab \geq \f{a + c}{b+d} \geq \f cd,
\end{equation}
it suffices to show that
\[
\frac{ 2\sigma_{s'_{k-1}}(\Xi')}{ (N' - s'_{k-1} )s'_{k-1} } \le \frac{d_1 - 1}{N - a_1}.
\]

If $m_2 > 0$, then the first term of $\Xi'$ is $d_1 - 2$, and our induction hypothesis gives that
\[
\frac{ 2\sigma_{s'_{k-1}}(\Xi')}{ (N' - s'_{k-1} )s'_{k-1} } \le \frac{d_1 - 2}{N' - m_2} = \frac{d_1 - 2}{N - 2m_1 - m_2}.
\]
We must therefore show that
\[
\frac{d_1 - 2}{N - 2m_1 - m_2} \le \frac{d_1 - 1}{N - a_1}.
\]
This simplifies to
\[
N \ge m_1 d_1 + m_2(d_1 - 1) = a_1 d_1 + a_2 d_2,
\]
which follows from the definition of $N$.

On the other hand, if $m_2 = 0$ we must have $m_3 > 0$, so in this case our induction hypothesis gives
\[
\frac{ 2\sigma_{s'_{k-1}}(\Xi')}{ (N' - s'_{k-1} )s'_{k-1} } \le \frac{d_1 - 3}{N' - m_3} = \frac{d_1 - 2}{N - 2m_1 - m_3}.
\]
It therefore suffices to show that
\[
\frac{d_1 - 2}{N - 2m_1 - m_3} \le \frac{d_1 - 1}{N - a_1},
\]
and this simplifies to $2N \ge m_1(d_1+1) + m_3(d_1 - 1)$.  This again follows from the definition of $N$, and the observation that if $m_3 > m_1$ then $m_3 = m_1 + a_2$ and $d_2 = d_1 - 2$.
\end{proof}




\subsection{Uncertainty Principle}

In Corollary \ref{cor:ASL2toGKdim} and Section \ref{sec:matrixcoefficientdecay} respectively, we established that the GK-dimension and matrix coefficient decay of unitarizable $\pi$ are both computable in terms of the augmented Arthur-$\SL_2$. This lets us compare the two invariants by some elementary inequality work. 

\begin{dfn}
   
Choose $N$ and let $d_{\max} = \frac12N(N-1)$. Define a normalized GK-parameter
\[
g(\pi) := 1 - d_{GK}(\pi)/d_{\max}
\]
for irreducible representations $\pi$ of $\GL_N(F)$. 
\end{dfn}

The parameter $g(\pi)$ ranges between $0$ when $\pi$ is generic and $1$ when $\pi$ is a character. 

\begin{thm}[``Uncertainty  Principle'']\label{thm:atypemctogk}
For all Arthur-type irreducible representations $\pi$ of $\GL_N(F)$:
\[
g(\pi) \leq 1 - \f2{p(\pi)} \leq g(\pi)^{1/2}. 
\]

\end{thm}

\begin{proof}

Let $A(\pi) = [d_1, \dotsc, d_r]$ with $d_1$ the largest entry. By Corollary \ref{cor:ASL2toGKdim} and some algebra,
\[
g(\pi) = \f1{N(N-1)} \sum_{i=1}^r d_i(d_i - 1). 
\]

For the lower bound, Proposition \ref{prop:matrixdecay} gives
\[
1 - \f2{p(\pi)} \geq \f{d_1-1}{N-1},
\]
since $(d_1 - 1)/2$ is the largest element of $\Xi(\pi)$. On the other hand,
\[
g(\pi) = \sum_{i=1}^r \f{d_i}N \f{d_i - 1}{N-1} \leq \sum_{i=1}^r \f{d_i}N \f{d_1 - 1}{N-1} = \f{d_1 - 1}{N-1}.
\]

For the upper bound, assume there are $m$ many $d_1$'s in $A(\pi)$. Then
\[
g(\pi) \geq \f{md_1(d_1 - 1)}{N(N-1)}.
\]
Corollary \ref{cor:mcdecaysimple} gives that
\[
1 - \f2{p(\pi)} = \f{d_1 - 1}{N-m},
\]
and the inequality 
\[
\lf(\f{d_1 - 1}{N-m} \ri)^2 \leq \f{md_1(d_1 - 1)}{N(N-1)}
\]
always holds under the condition $m \leq N/d_1$. 
\end{proof}

Figure \ref{fig:boundtightness} plots $2d_{\max}/p(\pi)$ against $d_{GK}(\pi)$ for all possible $A(\pi)$ when $N = 50$ together with the bounds from Theorem \ref{thm:atypemctogk}. As can be inspected, both bounds (particularly the square-root one) are quite tight.  

\begin{figure}
    \centering
    \includegraphics[width=0.8\linewidth]{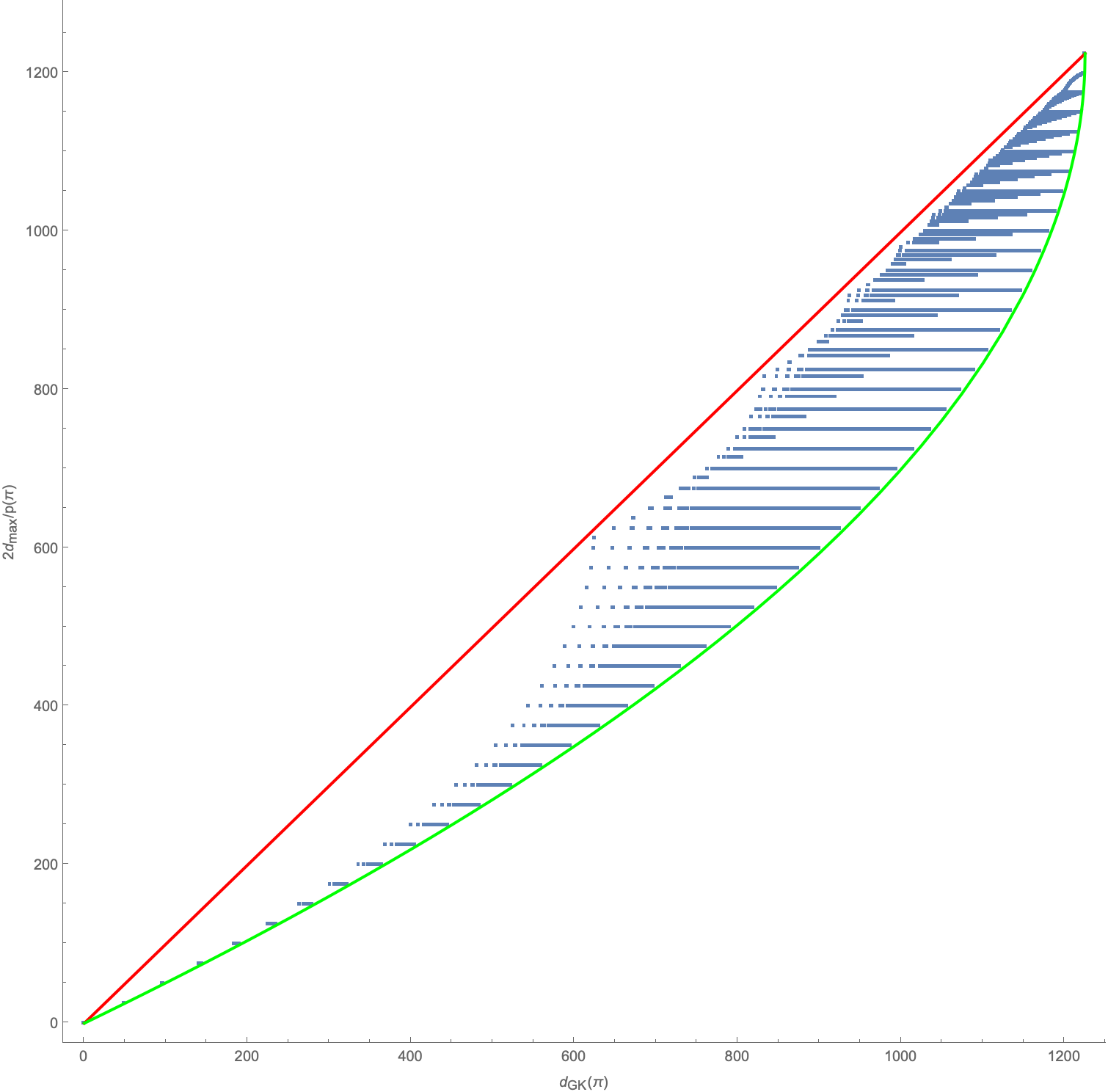}
    \caption{Tightness of Bounds in Theorem \ref{thm:atypemctogk} with $N = 50$.}
    \label{fig:boundtightness}
\end{figure}

Unitarizable representations satisfy a slightly weaker upper bound:

\begin{thm}\label{thm:unitarymctogk}
For all unitarizable irreducible representations $\pi$ of $\GL_N(F)$,
\[
g(\pi) \leq 1 - \f2{p(\pi)} \leq g(\pi)^{1/2} + \f2N. 
\]
\end{thm}

\begin{proof}
By Theorem \ref{thm:unitarizable}, a unitarizable representation is of the form
\[
\pi  = \bigoplus_{i=1}^{r_1}\rho_i[a_i][d_i] \oplus \bigoplus_{i=r_1+1}^{r_2} ( |\cdot|^{y_i} \rho_i[a_i][d_i] \oplus |\cdot|^{-y_i} \rho_i[a_i][d_i])
\]
for $0 < y_i < 1/2$. Define the corresponding Arthur-type representation
\[
\pi^- := \bigoplus_{i=1}^{r_1}\rho_i[a_i][d_i] \oplus \bigoplus_{i=r_1 + 1}^{r_2} (\rho_i[a_i][d_i] \oplus \rho_i[a_i][d_i]).
\]
Corollary \ref{cor:ASL2toGKdim} gives $g(\pi) = g(\pi^-)$. 

For the lower bound, the multiset of elements in $\Xi(\pi)$ is symmetric around $0$. Therefore, Proposition \ref{prop:matrixdecay} can be reduced to 
\[
1 - \f2{p(\pi)} = \max_{1 \leq i \leq N/2} \f{2 \sigma_i(\Xi(\pi))}{i(N-i)}  \geq \f{d_1 - 1}{N - a_1} \geq 1- \f2{p(\pi^-)}. 
\]
For the upper bound,
\[
1 - \f2{p(\pi)} < \max_{1 \leq i \leq N/2} \f{2 \sigma_i(\Xi(\pi^-)) + i}{i(N-i)} \leq 1 - \f2{p(\pi^-)} + \f2N
\]
into which we substitute the upper bound in Theorem \ref{thm:atypemctogk}. 
\end{proof}

\begin{rmk}
The ``worst case'' in the upper bound of Theorem \ref{thm:unitarymctogk} is when all $d_i = 1$ and the $x_i$ approach the boundary $x_i = \pm 1/2$. At this boundary $g(\pi) = 0$ and $1 - 2/p(\pi) = 2/N$.
\end{rmk} 

Combining the results of this section with Theorem \ref{thm:mainfixedvector} and Corollary \ref{cor:maincoefficientsalt} respectively gives bounds on fixed-vector growth and HCH-coefficients in terms of matrix coefficient decay:

\begin{cor}
For all $\eps > 0$, there is a uniform constant $C_\eps := C_{\eps, N, F}$ such that for any smooth irreducible representation $\pi$ of $\GL_N(F)$ with $p(\pi) \geq p_0$:
\[
\dim(\pi^{K_\ell}) \leq C_\eps q^{\ell \left(N(N-1) \lf(1 -  \lf( 1 - \f 2{p_0}\ri)^2 \ri)+\eps \right)}
\]
if $\pi$ is Arthur-type and
\[
\dim(\pi^{K_\ell}) \leq C_\eps q^{\ell \left( N(N-1) \lf(1 -  \lf( 1 - \f 2{p_0} - \f2N\ri)^2 \ri)+\eps\right) }
\]
if $\pi$ is only unitarizable. 
\end{cor}

\begin{cor}
For all $\eps > 0$, there is a uniform constant $C_\eps := C_{\eps, N, F}$ such that the Harish-Chandra--Howe coefficients of any smooth irreducible representation $\pi$ of $\GL_N(F)$ with $p(\pi) \geq p_0$ 
 are bounded by
\[
|c_O(\pi)| \leq C_\eps q^{\ell(\pi) \lf(N(N-1)\lf(1 -  \lf( 1 - \f 2{p_0}\ri)^2\ri) - \f12 \dim O + \eps \ri) } 
\]
if $\pi$ is Arthur-type and
\[
|c_O(\pi)| \leq C_\eps q^{\ell(\pi) \lf(N(N-1)\lf(1 -  \lf( 1 - \f 2{p_0} - \f2N\ri)^2\ri) - \f12 \dim O + \eps \ri) } 
\]
if $\pi$ is only unitarizable.
\end{cor}

\bibliographystyle{amsalpha}
\bibliography{Tbib}

\end{document}